\documentclass[reqno,11pt,a4paper]{article}
\usepackage[utf8]{inputenc}
\usepackage[english]{babel}
\usepackage[textheight=24 true cm, totalwidth=16.5 true cm,footskip=1.7cm]{geometry}
\usepackage{amsthm,amsmath,amssymb,dsfont}
\usepackage{tikz}
\usepackage{hyperref}
\hypersetup{colorlinks=true,
    linkcolor=blue,
    filecolor=magenta,
    urlcolor=cyan,
    citecolor=blue,
    pdfpagemode=FullScreen,
    }
\usepackage[capitalise]{cleveref}

\def\llbracket{[\![}
\def\rrbracket{]\!]}
\def\2stepset#1#2#3#4#5#6#7#8{%
  \begin{picture}(20,20)(-10,-10)
    \put(0,0){\ifx1#1\thicklines\let\x\vector\else\thinlines\let\x\line\fi\x(-1,-1){10}}
    \put(0,0){\ifx1#2\thicklines\let\x\vector\else\thinlines\let\x\line\fi\x(0,-1){10}}
    \put(0,0){\ifx1#3\thicklines\let\x\vector\else\thinlines\let\x\line\fi\x(1,-1){10}}
    \put(0,0){\ifx1#4\thicklines\let\x\vector\else\thinlines\let\x\line\fi\x(-1,0){10}}
    \put(0,0){\ifx1#5\thicklines\let\x\vector\else\thinlines\let\x\line\fi\x(1,0){10}}
    \put(0,0){\ifx1#6\thicklines\let\x\vector\else\thinlines\let\x\line\fi\x(-1,1){10}}
    \put(0,0){\ifx1#7\thicklines\let\x\vector\else\thinlines\let\x\line\fi\x(0,1){10}}
    \put(0,0){\ifx1#8\thicklines\let\x\vector\else\thinlines\let\x\line\fi\x(1,1){10}}
  \end{picture}}
\renewcommand{\geq}{\geqslant}
\renewcommand{\leq}{\leqslant}

\newtheorem{definition}{Definition}
\newtheorem{exa}[definition]{Example}
\newtheorem{remark}[definition]{Remark}
\newtheorem{theorem}[definition]{Theorem}
\newtheorem{proposition}[definition]{Proposition}
\newtheorem{lemma}[definition]{Lemma}
\newtheorem{corollary}[definition]{Corollary}
\numberwithin{equation}{section}
\newenvironment{ack}{\bigskip\noindent{\em Acknowledgements.}}{}

\def\cA{\mathcal{A}}
\def\cB{\mathcal{B}}
\def\cC{\mathcal{C}}
\def\cD{\mathcal{D}}
\def\cE{\mathcal{E}}
\def\cS{\mathcal{S}}
\def\C{\mathbb{C}}
\def\N{\mathbb{N}}
\def\Q{\mathbb{Q}}
\def\R{\mathbb{R}}
\def\Z{\mathbb{Z}}
\def\l{\left}
\def\r{\right}
\def\wtilde{\widetilde}

\title{Singular walks in the quarter plane and Bernoulli numbers}

\author{Alin~Bostan\footnote{Inria, Sorbonne Universit\'e, LIP6, Paris, France.}, Lucia~Di Vizio\footnote{CNRS, Universit\'e Paris-Saclay, UVSQ, Laboratoire de mathématiques de Versailles, Versailles, France.}
 and Kilian~Raschel\footnote{CNRS, Laboratoire Angevin de Recherche en Math\'ematiques, Universit\'e d'Angers, Angers, France.
\\
AB is partially supported by the French--Austrian project EAGLES (ANR-22-CE91-0007 \& FWF I6130-N). \\
KR is partially supported by the RAWABRANCH project ANR-23-CE40-0008, funded by the French National Research Agency, and by Centre Henri Lebesgue, programme ANR-11-LABX-0020-01.}}

\begin{document}

\maketitle

\begin{abstract}
We consider singular (aka genus-zero) walks in the quarter plane and their associated generating functions $Q(x,y,t)$, which enumerate the walks starting from the origin, of fixed endpoint (encoded by the spatial variables $x$ and $y$) and of fixed length (encoded by the time variable $t$). We first prove that the previous series can be extended up to a universal value of $t$ (in the sense that this holds for all singular models), namely $t=\frac{1}{2}$, and we provide a probabilistic interpretation of $Q(x,y,\frac{1}{2})$. As a second step, we refine earlier results in the literature and show that $Q(x,y,t)$ is indeed differentially transcendental for any $t\in(0,\frac{1}{2}]$. Moreover, we prove that $Q(x,y,\frac{1}{2})$ is  strongly differentially transcendental.
As a last step, we show that for certain models the series expansion of $Q(x,y,\frac{1}{2})$ is directly related to Bernoulli numbers. This provides a second proof of its strong differential transcendence.
\end{abstract}

\setcounter{tocdepth}{1}
\small
\tableofcontents
\normalsize

\section{Introduction}

We consider the classical models of quarter plane walks in combinatorics: Given a
finite set $\cS\subset\Z^2$ of prescribed steps
(see examples in Table~\ref{tab:list}), describing the possible moves of the walk, one of the main objectives is to enumerate, either exactly or asymptotically, the number of walks starting at the origin $(0,0)$, staying in the quarter plane $\N^2=\{0,1,2,\ldots\}^2$, with jumps in the step set $\cS$, of length $n$, and ending at the point $(i,j)$; see Figure~\ref{fig:example_path} for an example of path.
Denote by $\#_{\cS}\{(0,0)\stackrel{n}{\to} (i,j)\}$
the number of such walks.
The associated generating function is
\begin{equation}
\label{eq:def_generating_function}
   Q_\cS(x,y,t) = \sum_{i,j,n\geq0}\#_{\cS}\{(0,0)\stackrel{n}{\to} (i,j)\} x^iy^jt^n \in\N[x,y]\llbracket t\rrbracket\subset\Z\llbracket x,y,t\rrbracket.
\end{equation}
We focus on the step set in the Table \ref{tab:list}.
In our first result we prove that the function $Q(x,y)$ converges for $|x|,|y|<1$ and $|t|<\frac{1}{2}$ and that
it can be defined as a series in $\mathbb R\llbracket x,y\rrbracket$
up to $ t=\pm\frac{1}{2}$.
Furthermore, at $t=\frac{1}{2}$, its coefficients are all rational.

\medskip

\begin{table}[ht!]\centering
$\begin{array}{ccccc}
    \qquad\textcolor{red}{\2stepset00100101}\qquad
        & \qquad\textcolor{red}{\2stepset00101110}\qquad
        & \qquad\textcolor{red}{\2stepset00101111}\qquad
        & \qquad\textcolor{red}{\2stepset00100110}\qquad
        & \qquad\textcolor{red}{\2stepset00100111} \\
    \qquad\mathcal{A}\qquad
        & \qquad\mathcal{B}\qquad
        & \qquad\mathcal{C}\qquad
        & \qquad\mathcal{D}\qquad
        & \qquad\mathcal{E}
\end{array}$
\begin{quote}
\caption{The five main models considered in this work, called singular models. Among them, the first three
$\mathcal{A} = \{\textsf{NW},             \textsf{NE},             \textsf{SE}\},$
$\mathcal{B} = \{\textsf{NW}, \textsf{N},              \textsf{E}, \textsf{SE}\},$
$\mathcal{C} = \{\textsf{NW}, \textsf{N}, \textsf{NE}, \textsf{E}, \textsf{SE}\}$
are symmetric models with respect to the first diagonal,
while the last two
$\mathcal{D} = \{\textsf{NW}, \textsf{N},                          \textsf{SE}\},$
$\mathcal{E} = \{\textsf{NW}, \textsf{N}, \textsf{NE},             \textsf{SE}\}$
are not.}
\end{quote}
\label{tab:list}
\end{table}

\noindent
Then we prove our main theorem, namely:
\begin{theorem}[{See Theorem \ref{thm:main-Galois}}]
\label{thm:main-Galois-intro}
For any model $\cS$ in the list $\mathcal{A}$, $\mathcal{B}$, $\mathcal{C}$, $\mathcal{D}$, $\mathcal{E}$ and
for any $t\in\bigl(0,\frac 1 2\bigr]$ the formal power series
$Q(x,0)\equiv Q_\cS(x,0,t)$ (resp.\ $Q(0,y)$, $Q(x,y)$)
is differentially transcendental over $\C(x)$
(resp.\ $\C(y)$, $\C(x,y)$).
For $t=\frac 12$, $Q(x,0)$ (resp.\ $Q(0,y)$, $Q(x,y)$)
is strongly differentially transcendental.
\end{theorem}

We say that a formal power series with complex coefficients in $\C\llbracket x\rrbracket$ is differentially transcendental over $\C(x)$
(resp.\ strongly differentially transcendental), if it does not satisfy any
algebraic differential equation with coefficients in $\C(x)$ (resp.\ in a field of meromorphic functions in a neighborhood of $0$).

\paragraph*{State of the art.}
In combinatorics, scholars organize enumerative data as coefficients of formal power series, usually called \emph{ordinary generating functions}, OGF, for short.
In the impossibility of an exact formula for these coefficients, one usually studies their asymptotics and the singularities of their OGFs, or
tries to establish whether this latter series is solution of a certain type of functional equations.
It is particularly useful to understand whether an OGF is \emph{D-finite}, i.e., solution of a linear differential equation with rational functions coefficients:
in this case the coefficients of the OGF are P-recursive, i.e., they satisfy particular linear recurrence relations.
The lesson of the applications of Galois theory of functional equations to combinatorics is that it is frequently easier to prove that an OGF \emph{is not} a solution to a (possibly non-linear)\ differential equation with rational function coefficients, rather than to prove that it is not D-finite, although the latter condition is weaker.
It has indeed been proven that many OGFs are not only non-D-finite, but even \emph{differentially transcendental}, i.e., they are not solutions of any algebraic differential equation %
(we will be more precise soon).

\par
In 2003, Klazar \cite{Klazar03} proved that the OGF of Bell numbers (which count the number of partitions of a set of given cardinality)\ is in fact differentially transcendental over any field of meromorphic functions in a neighborhood of~$0$.
In a  previous work \cite{bostan-divizio-raschel-BELL}, we proved that Klazar's result is actually an instance of a rather
common phenomenon and we called such OGFs \emph{strongly differentially transcendental}. While mathematical intuition says that being strongly differentially transcendental is stronger than being merely differentially transcendental, it is not immediately clear what kind of information is encoded in such a property.

\par
Over the past two decades, the classical models of walks in the quarter plane introduced above
have attracted considerable attention from the mathematical community. The main questions are: First, is it possible to solve this enumeration problem in closed form (meaning, as elementary functions in the sense of Liouville)? If not, asymptotically (e.g., when the path length~$n$ tends to infinity)?
What is the ``complexity'' of the generating function~\eqref{eq:def_generating_function}, as a function of the step set~$\cS$? In other words, can the function $Q_\cS(x,y,t)$ be rational, algebraic, D-finite or differentially algebraic?

Although these questions can be addressed \textit{a priori} for any step set~$\cS$, it has been shown by Bousquet-M\'elou and Mishna in~\cite{BoMi10} and in subsequent works that the class of nearest neighbor walks already gives rise to an impressive variety of possible behaviors.

According to \cite{BoMi10}, the set of nearest neighbor walks can be divided into two subclasses, the \emph{singular} and the \emph{non-singular} models. By definition, singular models are those for which the step set $\cS$ is contained in a linear half-plane. Bousquet-M\'elou and Mishna proved that there are exactly five singular models that are truly two-dimensional, not equivalent, and can cross the quarter plane on both axes.
They are here called $\mathcal{A}$--$\mathcal{E}$
(following the convention of \cite{MeMi14}),
and they are shown in Table~\ref{tab:list}.
Note that, unlike models $\mathcal{D}$ and $\mathcal{E}$, the step sets of models $\mathcal{A}$, $\mathcal{B}$ and
$\mathcal{C}$ are symmetric.

These models are sometimes called \emph{directed} (see \cite{MiRe09}), in the sense that they cannot return to the origin once they have left it; see Figure~\ref{fig:example_path} for an example of a path, see also Figure~\ref{fig:recursive_cons}.
Singular models are also called
\emph{genus-zero models} (see \cite[Chap.~6]{FIM17}) for the following reason: in all five cases appearing on Table~\ref{tab:list}, defining the
characteristic (generating/inventory)\
polynomial to be the Laurent polynomial
\begin{equation}
    \label{eq:def_inventory}
    {\chi_{\cS}(x,y) := {\sum_{(i,j)\in\cS}x^{i}y^{j}}},
\end{equation}
the kernel polynomial
	\begin{equation}
	\label{eq:kernel}
	{K_{\cS}(x,y,t) := xy\bigl(1-t \cdot \chi_{\cS}(x,y)\bigr)}
	\end{equation}
is biquadratic in $x$ and $y$,
irreducible in $\mathbb{C}(t)[x,y]$,
which defines a pencil of algebraic curves of genus zero\footnote{This can be seen graphically using an old theorem of Henry Frederick Baker~\cite{Baker1893} (see also~\cite{Beelen09}), which says that the genus of a plane curve defined by a bivariate polynomial is upper bounded by the number of grid points inside the Newton polygon of the polynomial. In our case, the corresponding Newton polygons do not contain any grid points in their interiors.} for generic~$t$. In what follows, most of the time, $\cS$ will be taken as one of the five step sets listed in Table~\ref{tab:list}. When no ambiguity can arise, we will simplify our notation by removing the step set $\cS$ and the variable $t$; therefore $K(x,y)$ and $Q(x,y)$ will stand for $K_\cS(x,y,t)$ and $Q_\cS(x,y,t)$, respectively.

\begin{figure}[ht!]
    \centering
    \includegraphics[scale=0.45]{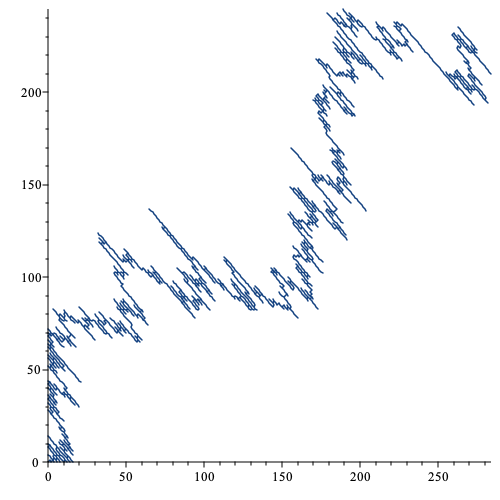}
    \begin{quote}
    \caption{An example of path for the model $\cA$ (10,000 steps). According to the classical law of large numbers, almost every path tends to infinity along the first diagonal, with antidiagonal ``Brownian'' fluctuations.}
    \end{quote}
    \label{fig:example_path}
\end{figure}

In the probabilistic setting (i.e., when one specializes $t=\frac{1}{\#\cS} $, where $\#\cS$ denotes the cardinality of the step set $\cS$), singular models were first identified in the book \cite[Chap.~6]{FIM17}. More precisely, genus-zero models are characterized in \cite[Thm.~6.1.1 and \S6.4]{FIM17}, where various properties of singular models are derived, such as functional equations and explicit expressions for the generating functions, boundary value problems, uniformizations of the associated kernel curve~\eqref{eq:kernel}, etc.
See~\cite{DracDrey,FaIa-21} for further insights into possible uniformizations.

In the combinatorial literature, singular models first appeared in Bousquet-M\'elou and Mishna's classification~\cite{BoMi10} of walks with small steps in the quarter plane. They proved that all singular models admit an infinite group (in a sense, this group captures the symmetries of the step set $\cS$). See also \cite{Mishna-09} for a classification of quarter plane walks when the step set has cardinality~$3$.

In \cite{MiRe09}, Mishna and Rechnitzer consider the models $\cA$ and $\cD$ and express the generating functions $Q(x,0,t)$, $Q(0,y,t)$ and $Q(x,y,t)$ as infinite sums of algebraic functions. As a consequence, they prove that $Q(1,1,t)$ is not D-finite, i.e., it does not satisfy any linear differential equation in~$\frac{d}{dt}$ with polynomial coefficients in~$t$. The main argument is to show that $Q(1,1,t)$ has an infinite number of singularities, which is classically incompatible with D-finiteness. Melczer and Mishna extended these results to the models $\cB$, $\cC$ and $\cE$~\cite{MeMi14}.

Singular models are further considered in the work \cite{DHRS0} by Dreyfus, Hardouin, Roques and Singer. The authors show that if $0 < t < \frac{1}{\#\cS}$ is a \emph{transcendental} number, then $Q(x, y, t)$ is $x$- and $y$-differentially transcendental. As in the probabilistic literature, weights are allowed on the steps in their article~\cite{DHRS0}. The main technical novelty in~\cite{DHRS0} is the use of the Galois theory of $q$-difference equations, following their work~\cite{DHRS1} in the genus-one setting. Note that in~\cite{DHRS0} the weights are normalized so that their sum is~$1$ (thus they define transition probabilities); comparing their results with ours therefore requires changing the variable $t\mapsto \frac{t}{\#\cS}$, in other words their assumption $t\in(0,1)$ is equivalent to $t\in (0,\frac{1}{\#\cS})$ in our notation.

Finally, returning to the probabilistic literature, Janse van Rensburg, Prellberg and Rechnitzer \cite{Prellberg-08} consider self-avoiding singular walks in the quarter plane, compute their growth constants and prove functional equations for the associated generating functions. In \cite{Po-11} Poznanovi{\'c} obtains an interesting bijection to prove combinatorially some results in~\cite{Prellberg-08}. In~\cite{HoRaTa-23}, Hoang, Tarrago and the last author of this article compute positive harmonic functions with Dirichlet boundary conditions for singular walks (with possible large jumps). This last paper also contains asymptotic estimates of the Green function and the description of the Martin boundary.

\paragraph*{Our contributions and techniques.} Besides proving that $Q(x,y,t)$ converges for $|x|,|y|<1$ and $|t|<\frac 1 2$, our first result is the following:

\begin{theorem}[{See Theorem~\ref{thm:cont_1/2} below}]
For any model in Table~\ref{tab:list}, the power series
\begin{equation*}
    Q(x,y,\tfrac{1}{2}) = \sum_{i,j,n\geq0}\frac{\#_{\cS}\{(0,0)\stackrel{n}{\to} (i,j)\}}{2^n} x^iy^j\in\mathbb Q\llbracket x,y\rrbracket
\end{equation*}
is well defined, meaning that all its coefficients are finite and define rational numbers.
\end{theorem}

In order to prove the above theorem, we use, inspired by~\cite{MiRe09}, mainly a combinatorial approach to recursively compute the number of paths
$\#_{\cS}\{(0,0)\stackrel{n}{\to} (i,j)\}$. Obviously, any path of a singular model can be decomposed as a sequence of walks on antidiagonal segments
\begin{equation}
\label{eq:segments-intro}
    S_k = \{(i,j)\in \mathbb Z^2: i\geq 0,\, j\geq 0,\, i+j=k\}=\{(k,0),(k-1,1),\ldots,(0,k)\},
\end{equation}
followed by diagonal jumps, see Figure~\ref{fig:recursive_cons}, and also Figure~\ref{fig:example_path}. This decomposition can be used to derive a fairly simple expression for the number of walks, using basic linear algebra. See Section~\ref{sec:well_defined}.

In Section~\ref{sec:main},
we shall use that the generating function $Q(x,y)\equiv Q_\cS(x,y,t)$ in \eqref{eq:def_generating_function} satisfies the functional equation
\begin{equation}
\label{eq:kernel-funct-eq}
    K(x,y)Q(x,y)=xy-tx^2Q(x,0)-ty^2Q(0,y),
\end{equation}
with the kernel $K(x,y)$ as in~\eqref{eq:kernel}. By manipulating this equation as in~\cite{FIM17} and~\cite{DHRS0}, but with slightly different parametrizations,
we get a new functional equation satisfied by the \emph{section} $Q(x,0)$.
This new equation is a $q$-difference equation for $t\in\bigl(0,\frac{1}{2}\bigr)$
and a finite difference equation for $t=\frac{1}{2}$.
Using the Ishizaki-Ogawara theorem (which we recall in Theorem~\ref{thm:Ogawara})
and Theorem~\ref{thm:OurTheorem;)} (see \cite{bostan-divizio-raschel-BELL})
we prove Theorem~\ref{thm:main-Galois-intro}.
It extends the main theorem of \cite{DHRS0}, which is only proven for transcendental values of the variable $t$, and only for $t\in (0,\frac{1}{\#\cS})$. Note that the paper \cite{DHRS0} considers models with weights.
We recover their result in full generality only for the models with step sets of cardinality $3$ (namely $\cA$ and $\cD$), see Section~\ref{subsec:weights}, as we have three degrees of freedom in our parameters.

Let us give some heuristics about the change in behavior at $t=\frac{1}{2}$ implied by Theorem~\ref{thm:main-Galois-intro}.
First, one could argue that if a $q$-difference equation associated with a homography with two fixed points degenerates into a
finite difference equation, whose operator has only one fixed point, the confluence of two special points
is expected to increase the complexity of the solution.
We can propose another explanation. Namely
we show in Section~\ref{sec:probab} that $Q(x,y,t)$ admits a natural
probabilistic interpretation
as the Green function of a certain random walk, using the idea of the Cram\'er transform. As it turns out,
the latter transform becomes singular at $t=\frac{1}{2}$, so one can also interpret the
point $t=\frac{1}{2}$ as critical from this alternative, probabilistic point of view.

We conclude the paper with a small excursion in special function theory, showing that
we can compute explicitly the rational coefficients appearing in $Q(x,y,\tfrac{1}{2})$
in the case of the model $\mathcal{A}$. In fact, \emph{Bernoulli numbers} intriguingly show up:
\begin{theorem}
\label{thm:explicit_A}
For model $\mathcal{A}$ we have
\[ Q_\mathcal{A}(x,0,\pm \tfrac{1}{2}) =
\sum_{n \geq 0}  \frac{T_n}{4^n} \, x^{2n},
\]
where
$(T_n)_{n \geq 0} = \bigl(1, 2, 16, 272, 7936,\ldots\bigr)$
is the sequence of the \emph{tangent numbers} (\href{https://oeis.org/A000182}{A000182}).
Equivalently,
\[ Q_\mathcal{A}(x,0,\tfrac{1}{2}) = 2 \, \sum_{n \geq 0}   (2^{2n+2}-1) \frac{(-1)^n}{n+1} B_{2n+2} x^{2n},\]
where
$(B_n)_{n \geq 0} = \bigl( 1,-\frac12,\frac16,0, -\frac{1}{30},0,\frac{1}{42}, \ldots \bigr)$
is the sequence of Bernoulli numbers.
\end{theorem}

We derive a similar result for the models $\cB$ and $\mathcal{D}$, in Section~\ref{sec:case_1/2}.
As a direct consequence of Theorem~\ref{thm:explicit_A}, and using known results about Bernoulli numbers (e.g., recalled in \cite{bostan-divizio-raschel-BELL}), the power series
$Q_\cA(x,0, \tfrac{1}{2})$,
$Q_\cB(x,0, \tfrac{1}{2})$
 and
$Q_\cD(x,0, \tfrac{1}{2})$ are proved to be strongly differentially transcendental in a direct manner. This provides an alternative proof
of the $t=\frac12$ case of~\cref{thm:main-Galois-intro} for models $\cA$, $\cB$ and $\cD$
(\cref{thm:A-Bernoulli,thm:B-Bernoulli,thm:D-Bernoulli}).
Moreover, we prove that exponential versions of the power series $Q_\cA(x,0, \tfrac{1}{2})$ and $Q_\cD(x,0, \tfrac{1}{2})$ are D-algebraic; this can be viewed as yet another instance of the Pak-Yeliussizov Open Problem 2.4 in~\cite{Pak18}.

It would be interesting to find formulas of this kind for all the five singular models and actually establish whether
they come from the confluence of some kind of $q$-Bernoulli numbers, appearing for $t = q/(1+q^2)$ with $q \neq 1$.
Notice that the nature of $Q(1,1,t)$ still remains out of reach with these techniques.

\begin{ack}
Our warm thanks go to Mireille Bousquet-Mélou, for her interest in this project and for her insights shared with us at various stages of its preparation.
We are grateful to the participants of the working group ``Transcendence and Combinatorics'' for interesting discussions over the years, which stimulated this work. Finally, we thank the two referees for their very careful reading and for their useful feedback.
\end{ack}

\section{On the convergence of  \texorpdfstring{$\boldsymbol{Q(x,y,t)}$}{Q(x,y,t)} with respect to \texorpdfstring{$\boldsymbol{t}$}{t}}
\label{sec:well_defined}

As mentioned in the introduction, one of the original results of our paper concerns the function $Q(x,y,\tfrac{1}{2})$,
that is the evaluation at $t=\tfrac{1}{2}$ of the generating function $Q(x,y,t)$ introduced in \eqref{eq:def_generating_function}.
While the trivial estimate $\#_{\cS}\{(0,0)\stackrel{n}{\to} (i,j)\}\leq (\#{\cS})^n$ implies that $Q(x,y,t)$ is well defined and analytic for $\vert x\vert \leq 1$, $\vert y\vert \leq 1$ and $\vert t\vert<\frac{1}{\#\cS}$,
it is \textit{a priori} unclear that $Q(x,y,t)$ can be continued analytically for larger values of $\vert t\vert$.
This is the object of this section, in which we will prove
that $Q(x,y,t)$ converges for  $\vert x\vert,
\vert y\vert <1$ and $\vert t\vert<\frac{1}{2}$.
Moreover we prove that for any pair of integer $i,j\geq 0$ the series  $\sum_{n\geq 0}\#_{\cS}\{(0,0)\stackrel{n}{\to} (i,j)\}t^n$ converges at $t=\frac{1}{2}$ and takes rational values.
We will informally say that $Q(x,y,\frac{1}{2})$ is well defined and belongs to $\Q\llbracket x,y\rrbracket$.

\begin{theorem}
\label{thm:cont_1/2}
For any model in Table~\ref{tab:list}, the series $Q(x,y,t)$ converges for $|x|,|y|<1$ and $|t|<1/2$.
Moreover,
\begin{equation*}
    Q(x,y,\tfrac{1}{2}) = \sum_{i,j,n\geq0}\frac{\#_{\cS}\{(0,0)\stackrel{n}{\to} (i,j)\}}{2^n} x^iy^j\in\mathbb Q\llbracket x,y\rrbracket
\end{equation*}
is well defined, meaning that all its coefficients are finite and define rational numbers.
\end{theorem}

As explained by Mishna and Rechnitzer \cite{MiRe09}, singular walks naturally have a recursive construction, as directed steps $(1,0)$, $(1,1)$, $(0,1)$ followed by fluctuations with antidiagonal jumps $(-1,1)$ and $(1,-1)$; see Figure~\ref{fig:recursive_cons}. Reformulating some results of \cite{MiRe09}, we obtain in this section a matrix-product expression for the generating function $Q(x,y,t)$, which immediately leads to the following partial result (proved at the very end of this section, as a consequence of Propositions~\ref{prop:expression_matrix_A}, \ref{prop:expression_matrix_B_D} and \ref{prop:expression_matrix_C_E}):

\begin{proposition}
\label{prop:cont_1/2}
For any model in Table~\ref{tab:list}, the power series
$Q(x,y,\tfrac{1}{2})$
is well defined in $\R\llbracket x,y\rrbracket$, meaning that all its coefficients are finite.
\end{proposition}

We will complete the proof of Theorem~\ref{thm:cont_1/2}
by showing that the coefficients of $Q(x,y,\tfrac{1}{2})$ are rational (see Proposition~\ref{prop:rational_coeff} below).

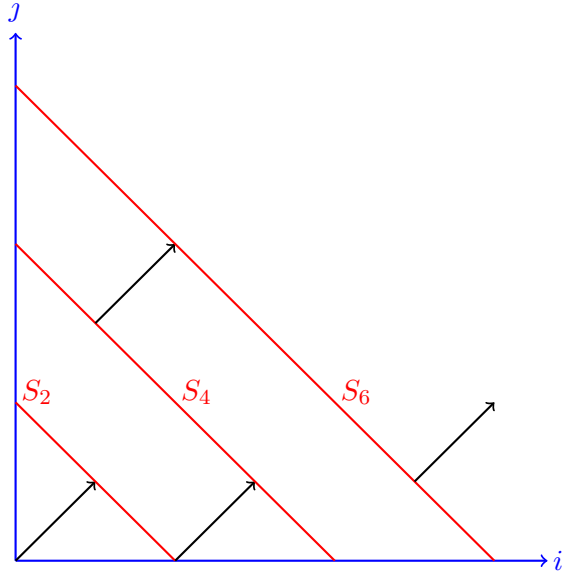
\begin{figure}
    \centering
\begin{tikzpicture}[scale=0.7]
    \draw[->,blue,thick] (0,0) -- (0,10);
    \draw[->,blue,thick] (0,0) -- (10,0);
    \draw[->,thick] (0,0) -- (1.5,1.5);
    \draw[-,red,thick] (0,3) -- (3,0);
    \draw[->,thick] (3,0) -- (4.5,1.5);
    \draw[-,red,thick] (6,0) -- (0,6);
    \draw[->,thick] (1.5,4.5) -- (3,6);
    \draw[-,red,thick] (9,0) -- (0,9);
    \draw[->,thick] (7.5,1.5) -- (9,3);
    \node at (0.4,3.2) {\textcolor{red}{$S_2$}};
    \node at (3.4,3.2) {\textcolor{red}{$S_4$}};
    \node at (6.4,3.2) {\textcolor{red}{$S_6$}};
    \node at (10.2,0) {\textcolor{blue}{$i$}};
    \node at (0,10.45) {\textcolor{blue}{$j$}};
   \end{tikzpicture}
   \begin{quote}
    \caption{Recursive construction of singular walks. In this example on model $\cA$, a walk can be decomposed as a sequence of jumps $(1,1)$ (black arrows)\ and antidiagonal jumps $(-1,1)$ and $(1,-1)$ (red). The segments $S_k$ are defined in \eqref{eq:segments}. See Figure~\ref{fig:example_path} for an example of a longer path.}
   \end{quote}
    \label{fig:recursive_cons}
\end{figure}

\subsection{Enumeration of walks on segments}

For any $k\geq 1$, define $S_k$ as the antidiagonal segment
\begin{equation}
\label{eq:segments}
    S_k = \{(i,j)\in \mathbb Z^2: i\geq 0,\, j\geq 0,\, i+j=k\}=\{(k,0),(k-1,1),\ldots,(0,k)\},
\end{equation}
see Figure~\ref{fig:recursive_cons}. For $P,Q\in S_k$ and $n\in\mathbb N$, define $w_k(P,Q,n)$ as the number of paths
with step set $\{(-1,1),(1,-1)\}$
in $S_k$ starting at~$P$, ending at~$Q$ and having length~$n$. We further introduce the generating function
\begin{equation}
\label{eq:GF_segments}
    W_k(P,Q,t):=\sum_{n\geq 0} w_k(P,Q,n) t^n.
\end{equation}
For $n\geq 2$, we consider the matrix
\begin{equation}
\label{eq:def_F_n}
    F_{n}=\begin{pmatrix}
1 & -t & 0 &\cdots &\cdots &  0 \\
-t & 1 & -t & 0 &\cdots &  0 \\
0&-t & 1 & -t & 0  & 0 \\
\vdots  & \ddots  & \ddots& \ddots & \ddots  & \vdots\\
\vdots  &  \cdots & 0 & -t& 1 & -t  \\
0 & \cdots  & \cdots& 0 & -t & 1
\end{pmatrix}^{-1}\in\mathcal M_{n,n}(\mathbb R(t)),
\end{equation}
which in Lemma~\ref{lem:rationality_1D_RW} below allows us to completely characterize the generating functions $W_k(P,Q,t)$ in \eqref{eq:GF_segments}.
Recall that the $n$-th Chebychev polynomial of the second kind $U_n(z) \in \mathbb{Q}[z]$ is defined via the generating function
$1/(1-2zt+t^2) = \sum_{n=0}^\infty U_n(z)t^n$, or alternatively
via the equality
$U_n(\cos \theta) \sin \theta = \sin\big((n + 1)\theta\big)$.
Then, we have the following well-known statement:

\begin{proposition}[{\cite[Thm.~3.2]{ka89}}]
\label{prop:coeff_F_n}
For $1\leq i,j\leq n$ and $\vert t\vert< \frac{1}{2\cos (\frac{\pi}{n+1})}$, let
\begin{equation*}
  f_{i,j} = \frac{(-1)^{i+j+1}}{t} \frac{U_{n-j}(-\frac{1}{2t}) U_{i-1}(-\frac{1}{2t})}{U_{n}(-\frac{1}{2t})}
  = \frac{1}{t} \frac{U_{n-j}(\frac{1}{2t}) U_{i-1}(\frac{1}{2t})}{U_{n}(\frac{1}{2t})}.
\end{equation*}
Then the $(i,j)$-entry of the matrix of $F_n$ is $f_{i,j}$ for $i\leq j$ and $f_{j,i}$ for $i>j$.
\end{proposition}

For an upcoming use, let us make explicit the matrix $F_n$ at $t=\frac{1}{2}$.
Since
$\sum_{n=0}^\infty U_n(-1)t^n = 1/(1+t)^2$,
we have that $U_n(-1)=(-1)^n(n+1)$, so $f_{i,j}=2\frac{i(n-j+1)}{n+1}$ and thus
\begin{equation}
\label{eq:def_F_n_inverse_at_1/2}
    F_n\left(\frac 12\right)=\frac{2}{n+1}\begin{pmatrix}
n & n-1 & n-2 &\cdots &2 &  1 \\
n-1 & 2(n-1) & 2(n-2) & \cdots &4 &  2 \\
n-2 & 2(n-2) & 3(n-2) &\cdots & 6  & 3 \\
\vdots  & \vdots  & \vdots& \ddots & \vdots  & \vdots\\
2  &  4 & 6 & \cdots& 2(n-1) & n-1  \\
1 & 2  & 3& \cdots & n-1 & n
\end{pmatrix}.
\end{equation}

We now state Lemma~\ref{lem:rationality_1D_RW}.
Although it is not new, we prove it briefly, as this gives us the opportunity to keep our paper self-contained and to introduce some useful notations. Lemma~\ref{lem:rationality_1D_RW} immediately follows from the transfer-matrix method, see e.g.\ \cite[Sec.~4.7]{St-12}, in particular Theorems~4.7.1 and 4.7.2.

\begin{lemma}
\label{lem:rationality_1D_RW}
Let $k\geq 1$. For $0\leq i\leq k$, denote by $P_i = (k-i,i)$ the points of $S_k$, see \eqref{eq:segments}. The following equality between matrices of generating functions holds:
\begin{equation}
\label{eq:identities_matrix_inverse}
    \bigl(W_k(P_i,P_j,t)\bigr)_{0\leq i,j\leq k}=F_{k+1}.
\end{equation}
Accordingly, for all $0\leq i,j\leq k$, the series $W_k(P_i,P_j,t)$ is rational and its closest singularity to $0$ (a pole of order one)\ is at $t=\pm \bigl(2\cos\bigl(\frac{\pi}{2+k}\bigr)\bigr)^{-1}$, with $\bigl(2\cos\bigl(\frac{\pi}{2+k}\bigr)\bigr)^{-1}>\frac{1}{2}$. In the neighborhood of $\bigl(2\cos\bigl(\frac{\pi}{2+k}\bigr)\bigr)^{-1}$ we have
\begin{equation}
    \label{eq:series_expansion_Wk}
    W_k(P_i,P_j,t) = \frac{c_{i,j,k}}{\bigl(2\cos\bigl(\frac{\pi}{k+2})\bigr)^{-1} -t}+O(1),
\end{equation}
with $c_{i,j,k}>0$. The series $W_{k}(P_i,P_j,t)$ being even or odd, an expansion similar to \eqref{eq:series_expansion_Wk} holds in the neighborhood of $-\bigl(2\cos\bigl(\frac{\pi}{k+2})\bigr)^{-1}$.
\end{lemma}

\begin{proof}
We first look at the row vector of generating functions
\begin{equation*}
    \bigl(W_k(P_0,P_0,t),\ldots,W_k(P_k,P_0,t)\bigr)
\end{equation*}
with fixed ending point at $P_0$ and varying starting points in $S_k$. Using a first step decomposition of the walk, we easily obtain the identities
\begin{align*}
    W_k(P_0,P_0,t) &= 1 + t W_k(P_1,P_0,t),\\
    W_k(P_i,P_0,t) & = t W_k(P_{i-1},P_{0},t)+t W_k(P_{i+1},P_0,t), &1\leq i \leq k-1,\\
    W_k(P_k,P_0,t) & = tW_k(P_{k-1},P_0,t).
    \end{align*}
In terms of matrix product, this means that
\begin{equation*}
    \bigl(W_k(P_0,P_0,t),\ldots,W_k(P_k,P_0,t)\bigr) F_{k+1}^{-1} = (1,0,\ldots ,0),
\end{equation*}
the first vector of the canonical basis of $\mathbb R^{k+1}$. Similar computations show that
\begin{equation*}
    \bigl(W_k(P_0,P_i,t),\ldots,W_k(P_k,P_i,t)\bigr) F_{k+1}^{-1}
\end{equation*}
is the $(i+1)$-th vector of the canonical basis of $\mathbb R^{k+1}$, which proves Equation~\eqref{eq:identities_matrix_inverse}, since $F_{k+1}$ is symmetric.

We now prove the statement on the radius of convergence. For this we use the classical representation of the Chebychev polynomials as determinants of tridiagonal matrices, proving that
\begin{equation}
\label{eq:determinant_Cheb}
    \det (F_n^{-1}) = (-t)^n U_n\left(-\frac{1}{2t}\right).
\end{equation}
On the other hand, as $U_n(\cos\theta) = \frac{\sin((n+1)\theta)}{\sin \theta}$, the largest zero of $U_n(z)$ happens at $z = \cos(\frac{\pi}{n+1})$. As a consequence, the matrix $F_n$ is defined for $\vert t\vert< \frac{1}{2\cos (\frac{\pi}{n+1})}$, which corresponds to the statement after the change of variable $n=k+1$.

In order to prove \eqref{eq:series_expansion_Wk}, we perform a series expansion of $U_{n}$ around its first zero $\cos(\frac{\pi}{n+1})$, namely
\begin{equation*}
    U_n(z)=\frac{(n + 1)}{\sin(\frac{\pi}{n + 1})^2}
    \bigl(z-\cos(\tfrac{\pi}{n + 1})\bigr)+\cdots
\end{equation*}
and use Proposition~\ref{prop:coeff_F_n}. In this way, we obtain \eqref{eq:series_expansion_Wk}, with the following expression for the constant
\begin{equation*}
    c_{i,j,k}=\frac{\sin(\frac{\pi}{k+2})^2U_{k-j}\bigl(\cos(\frac{\pi}{k+2})\bigr)U_{i}\bigl(\cos(\frac{\pi}{k+2})\bigr)}{(k + 2)\cos(\frac{\pi}{k+2})},
\end{equation*}
on which we can read the positivity (recall that $\cos(\frac{\pi}{k+2})$ is the largest root of $U_{k+1}$ and that the roots of the Chebychev polynomials interlace, in such a way that Chebychev polynomials with index lower than $k+1$ are positive on the segment $[\cos(\frac{\pi}{k+2}),1]$).
\end{proof}

We apply these preliminary results to the generating functions associated to our five models, starting with model $\cA$.

\subsection{Preliminary results for
singular walks with step set \texorpdfstring{$\boldsymbol{\cA}$}{A}}

In this subsection, we focus on the model $\cA$, the first model in Table~\ref{tab:list}.
Notice that for model $\cA$ a walk can jump from one segment $S_k$ to another
only thanks to the step $(1,1)$: a walk starting in $(0,0)$ can only jump to $S_2$ and then either stay on $S_2$ or jump to $S_4$, and so on.
This immediately implies that
there are no walks ending on points of the segment $S_k$, when $k$ is odd.
More formally, let us fix $Q_k\in S_k=\{(k-i,i):i=0,\dots,k\}$ and consider the series
    \begin{equation}
    \label{eq:def_GF_A}
        \sum_{n\geq0} \#_{\cA}\{(0,0)\stackrel{n}{\to} Q_k\}t^n.
    \end{equation}
Then the series \eqref{eq:def_GF_A} above is identically zero for any odd value of the integer $k\geq 1$.
This proves the first (trivial)\ part of the following proposition, which is the key-point of this subsection.
To state its second part, we need to define the following sequence of matrices $(A_n)_{n\geq 3}$:
\begin{equation}
\label{eq:def_A_n}
    A_n=\begin{pmatrix}
0 & 1 & 0 &\cdots &\cdots & 0 \\
0 & 0 & 1 & 0 &\cdots & 0 \\
\vdots  & \vdots  & \ddots & \ddots& \ddots & \vdots  \\
0 & \cdots & \cdots& 0 & 1 & 0
\end{pmatrix}\in\mathcal M_{n-2,n}(\mathbb R).
\end{equation}

\begin{proposition}
\label{prop:expression_matrix_A}
If $k$ is odd, the series \eqref{eq:def_GF_A} is identically zero. Moreover, for any $k\geq 1$, one has
        \begin{equation*}
           \left(\sum_{n\geq0} \#_{\cA}\{(0,0)\stackrel{n}{\to} (2k-i,i)\}t^n\right)_{0\leq i\leq 2k}=t^kA_3 F_3 \cdots A_{2k+1}F_{2k+1},
        \end{equation*}
    with $A_{2k+1}$ and $F_{2k+1}$ as in \eqref{eq:def_A_n} and \eqref{eq:def_F_n}, respectively.
\end{proposition}

\begin{remark}
\normalfont
The expression given in Proposition~\ref{prop:expression_matrix_A} is robust in the following sense: if the boundary of the quarter plane were replaced by any other boundary, as in Figure~\ref{fig:other_boundaries}, then a very similar matrix-product expression for the generating function would hold. This observation applies more generally for all singular step sets $\mathcal S$.
\end{remark}

\begin{figure}
    \centering
\begin{tikzpicture}[scale=0.5]
    \draw[-,blue,thick] (0,0) -- (2,0);
   \draw[-,blue,thick] (2,0) -- (3,1);
   \draw[-,blue,thick] (3,1) -- (5,1);
   \draw[-,blue,thick] (5,1) -- (6,2);
   \draw[->,blue,thick] (6,2) -- (8,2);
   \draw[-,blue,thick] (0,0) -- (0,2);
   \draw[-,blue,thick] (0,2) -- (1,3);
   \draw[-,blue,thick] (1,3) -- (1,5);
   \draw[-,blue,thick] (1,5) -- (2,6);
   \draw[->,blue,thick] (2,6) -- (2,8);
   \end{tikzpicture}\qquad\qquad\qquad
   \begin{tikzpicture}[scale=0.5]
    \draw[-,blue,thick] (0,0) -- (2,0);
   \draw[-,blue,thick] (2,0) -- (3,1);
   \draw[-,blue,thick] (3,1) -- (4,0);
   \draw[-,blue,thick] (4,0) -- (5,2);
\draw[->,blue,thick] (5,2) -- (8,3);
\draw[-,blue,thick] (0,0) -- (-1,1);
\draw[-,blue,thick] (-1,1) -- (0,2);
\draw[-,blue,thick] (0,2) -- (2,3);
\draw[-,blue,thick] (2,3) -- (0,5);
\draw[-,blue,thick] (0,5) -- (-1,6);
\draw[-,blue,thick] (-1,6) -- (1,7);
\draw[->,blue,thick] (1,7) -- (2,8);
   \end{tikzpicture}
    \begin{quote}
    \caption{Proposition~\ref{prop:expression_matrix_A} can be adapted to apply to general domains, not only quarter planes. On the left, a domain with periodic boundary; on the right, an arbitrary domain.\label{fig:other_boundaries}}
    \end{quote}
\end{figure}
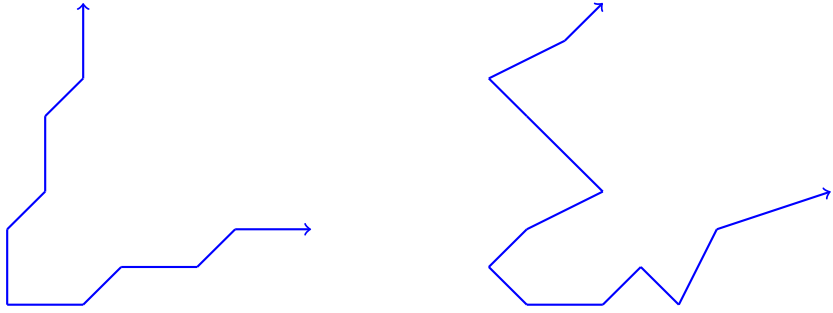

Proposition~\ref{prop:expression_matrix_A} is not new and is equivalent to a result proved in \cite{MiRe09}:
we will compare our formulation with the one contained in {\it loc.\ cit.}\ in Corollary~\ref{cor:rec_MiRe09} below.
To prove Proposition~\ref{prop:expression_matrix_A} we need the following lemma:

\begin{lemma}
\label{lemma:GF_A_rat_RC}
For even values of $k\geq 1$, the generating function \eqref{eq:def_GF_A} is rational and its first poles are at
$t=\pm\bigl(2\cos\bigl(\frac{\pi}{2+k}\bigr)\bigr)^{-1}$.
\end{lemma}

\begin{proof}
Using the obvious decomposition of a walk from $(0,0)$ to $Q_{2k}\in S_{2k}$ as a sequence of jumps $(1,1)$ followed by fluctuations on the segments $S_2,\ldots,S_{2k}$, see again Figure~\ref{fig:recursive_cons}, one has
   \begin{multline*}
       \#_{\cA}\{(0,0)\stackrel{n+k}{\to} Q_{2k}\} =\sum_{Q_2\in S_2,Q_4\in S_4,\ldots,Q_{2k-2}\in S_{2k-2}} \\\sum_{n_1+\cdots +n_k = n} w_2((1,1),Q_2,n_1)w_4(Q_2+(1,1),Q_4,n_2)\cdots w_{2k}(Q_{2k-2}+(1,1),Q_{2k},n_k).
   \end{multline*}
As a consequence,
   \begin{multline}
   \label{eq:decomp_GF}
       \sum_{n\geq0} \#_{\cA}\{(0,0)\stackrel{n}{\to} Q_{2k}\}t^{n} =\\t^k \sum_{Q_2\in S_2,Q_4\in S_4,\ldots,Q_{2k-2}\in S_{2k-2}} W_2((1,1),Q_2,t)\cdots W_{2k}(Q_{2k-2}+(1,1),Q_{2k},t).
   \end{multline}
Using now Lemma~\ref{lem:rationality_1D_RW},
one sees that the series \eqref{eq:def_GF_A}, written as in the expression above, is a finite sum of rational functions, with explicitly given radii of convergence.

To show that the singularity of $\sum_{n\geq0} \#_{\cA}\{(0,0)\stackrel{n}{\to} Q_{2k}\}t^{n}$ in \eqref{eq:decomp_GF} at $t=\bigl(2\cos\bigl(\frac{\pi}{2+2k}\bigr)\bigr)^{-1}$ is not removable, we use Equation~\eqref{eq:series_expansion_Wk} from Lemma~\ref{lem:rationality_1D_RW}. More specifically, we perform an expansion of \eqref{eq:decomp_GF} at $t=\bigl(2\cos\bigl(\frac{\pi}{2+2k}\bigr)\bigr)^{-1}$. We have
\begin{equation*}
    \sum_{n\geq0} \#_{\cA}\{(0,0)\stackrel{n}{\to} Q_{2k}\}t^{n} = \frac{C}{\bigl(2\cos(\frac{\pi}{2+2k})\bigr)^{-1} -t}+O(1),
\end{equation*}
where $C$ can be computed as follows, setting $t_0=\bigl(2\cos(\frac{\pi}{2+2k})\bigr)^{-1}$:
\begin{equation}
\label{eq:C_def}
    C=t_0^{k}
    \sum_{Q_2\in S_2,Q_4\in S_4,\ldots,Q_{2k-2}\in S_{2k-2}} W_2((1,1),Q_2,t_0)\cdots W_{2k-2}(Q_{2k-4}+(1,1),Q_{2k-2},t_0)\widetilde c_{Q_{2k-2},Q_{2k},k},
\end{equation}
where $\widetilde c_{Q_{2k-2},Q_{2k},k}=c_{i,j,2k}$ in \eqref{eq:series_expansion_Wk}, with the indices $i$ and $j$ corresponding respectively to $P_i=Q_{2k-2}+(1,1)$ and $P_j=Q_{2k}$. In \eqref{eq:C_def},
$C$ appears as a sum of terms which are all positive, therefore $C>0$.
The proof is complete.
\end{proof}

\begin{proof}[Proof of Proposition~\ref{prop:expression_matrix_A}]
Given $k\geq 1$, the matrix $A_{2k+1}$ describes how the walk jumps from the segment $S_{2k-2}$ to the next one, $S_{2k}$, and may thus be called a \emph{jump matrix}. More precisely, in the model $\cA$, if the walk exits the segment $S_{2k-2}$ at the point $(2k-2-i,i)$, then it will enter $S_{2k}$ at position $(2k-2-i,i)+(1,1) = (2k-1-i,i+1)$; see Figure~\ref{fig:recursive_cons}. From a matrix point of view, this corresponds to the following computation: given the vector $e_{i}^{2k-1}:=(0,\ldots,0,1,0,\ldots,0)\in \mathcal M_{1,2k-1}(\mathbb R)$ with the $1$ at position $i$, one has
\begin{equation*}
    e_{i}^{2k-1} A_{2k+1} = e_{i+1}^{2k+1}.
\end{equation*}
The proof of Proposition~\ref{prop:expression_matrix_A} is then immediate  from \eqref{eq:decomp_GF} and using Lemma~\ref{lem:rationality_1D_RW}.
\end{proof}

In this part we compare our results to that of Mishna and Rechnitzer \cite{MiRe09}.
They consider the generating function
\begin{equation}
\label{eq:Dk_MiRe09}
    D_{k}(y) = \sum_{i=0}^k \sum_{n\geq 0} \#_{\cA}\{(0,0)\stackrel{n}{\to} (k-i,i)\}t^n y^i \; \in\mathbb Q(t)[y] \cap \mathbb Q [y][[t]].
\end{equation}
It is clear that $D_0(y) = 1$ and $D_{2\ell+1}(y)=0$ for all $\ell \geq 0$.
Mishna and Rechnitzer further prove, with the notation $t=\frac{q}{1+q^2}$, that the following result holds.
\begin{corollary}[{\cite[Eq.~(45)]{MiRe09}}]
\label{cor:rec_MiRe09}
The generating function $D_k(y)$ satisfies the functional equation
\begin{equation}
\label{eq:rec_MiRe09}
    (q^{k+2}+1)(yq-1)(y-q)D_k(y) = q^3(y^{k+2}+1)D_{k-2}(q)-qy^2(q^{k+2}+1)D_{k-2}(y).
\end{equation}
\end{corollary}

We now show how we can deduce Corollary~\ref{cor:rec_MiRe09} from our Proposition~\ref{prop:expression_matrix_A}.
\begin{proof}[Proof of Corollary~\ref{cor:rec_MiRe09}]
We shall identify polynomials with row vectors. Therefore, the generating function $D_{2k}$ in \eqref{eq:Dk_MiRe09} is exactly the row vector appearing in the left-hand side of the main equation in Proposition~\ref{prop:expression_matrix_A}.

We shall first rewrite \eqref{eq:rec_MiRe09} in an equivalent way, simply by dividing the left- and the right-hand sides by $q(q^{k+2}+1)/t$, and then replacing $k$ by $2k$. We get
\begin{equation}
\label{eq:rec_MiRe09_bis}
    (ty^2-y+t)D_{2k}(y) = tq^2\frac{y^{2k+2}+1}{q^{2k+2}+1}D_{2k-2}(q)-ty^2D_{2k-2}(y).
\end{equation}
Using the identification between polynomials and row vectors, \eqref{eq:rec_MiRe09_bis} is an equality between row vectors of length $2k+3$.

Using now Proposition~\ref{prop:expression_matrix_A}, the left-hand side of \eqref{eq:rec_MiRe09_bis} may be written as
\begin{equation*}
    (ty^2-y+t)D_{2k}(y) = t^kA_3 F_3 \cdots A_{2k+1}F_{2k+1}\begin{pmatrix}
t & -1 & t &\cdots &\cdots & 0 \\
0 & t & -1 & t &\cdots & 0 \\
\vdots  & \vdots  & \ddots & \ddots& \ddots & \vdots  \\
0 & \cdots & \cdots& t & -1 & t
\end{pmatrix},
\end{equation*}
where the last matrix belongs to $\mathcal M_{2k+1,2k+3}(\mathbb R[t])$. The product of the last three matrices simplifies remarkably:
\begin{equation*}
   \bigl( t(F_{2k+1})_{\llbracket 2,2k\rrbracket,1} \vert 0\vert  -I_{2k-1} \vert 0\vert t(F_{2k+1})_{\llbracket 2,2k\rrbracket,2k+1} \bigr)\in\mathcal M_{2k-1,2k+3}(\mathbb R(t)),
\end{equation*}
where we use a column representation of the matrix, where $I_{2k-1}$ denotes the $(2k-1)\times(2k-1)$ identity matrix, and where we adopt the following notation: $(F_{2k+1})_{\llbracket 2,2k\rrbracket,1}$ is a column vector of length $2k-1$ formed by the
$(2,1)$-entry up to $(2k,1)$-entry of the matrix $F_{2k+1}$.

The last term in \eqref{eq:rec_MiRe09_bis} is
\begin{equation*}
    -ty^2D_{2k-2}(y) = -t^{k-1}A_3 F_3 \cdots A_{2k-1}F_{2k-1} t \bigl( 0\vert 0\vert I_{2k-1}\vert 0\vert 0 \bigr).
\end{equation*}
We can now state a matrix identity which is a lifting of \eqref{eq:rec_MiRe09_bis} at the matrix level
(that can be seen as a factorization by the common left factor $t^{k}A_3 F_3 \cdots A_{2k-1}F_{2k-1}$):
\begin{multline*}
    A_{2k+1}F_{2k+1}\begin{pmatrix}
t & -1 & t &\cdots &\cdots & 0 \\
0 & t & -1 & t &\cdots & 0 \\
\vdots  & \vdots  & \ddots & \ddots& \ddots & \vdots  \\
0 & \cdots & \cdots& t & -1 & t
\end{pmatrix} + \bigl( 0\vert 0\vert I_{2k-1}\vert 0\vert 0 \bigr) \\- qtA_{2k+1}F_{2k+1}
\begin{pmatrix}
1 & 0  &\cdots & 0 & 0 \\
0 & 0  &  & \vdots  & \vdots\\
\vdots   & \vdots&  & 0& 0  \\
0 & 0 & \cdots & 0 & 1
\end{pmatrix}=0.
\end{multline*}
This concludes the proof.
\end{proof}

\subsection{Preliminary results for the models \texorpdfstring{$\boldsymbol{\cB}$, $\boldsymbol{\cC}$, $\boldsymbol{\cD}$ and $\boldsymbol{\cE}$}{B, C, D and E}}

In this subsection we denote by $\cS$ any of the following step sets: $\cB$, $\cC$, $\cD$ or $\cE$, thereby extending the results to all models of Table~\ref{tab:list}.
As above, the number of $n$-length $\cS$-walks from $(0,0)$ to an arbitrary point $Q_k\in S_k$ is denoted by $\#_{\cS}\{(0,0)\stackrel{n}{\to} Q_k\}$. We begin with a statement analogous to Lemma~\ref{lemma:GF_A_rat_RC}.

\begin{lemma}
\label{lemma:GF_BCDE_rat_RC}
If $\cS\in\{\cB,\cC,\cD,\cE\}$,
then for any $k\geq 1$ and any $Q_k\in S_k$, the generating function
\begin{equation}
\label{eq:def_GF_L}
    \sum_{n\geq0} \#_{\cS}\{(0,0)\stackrel{n}{\to} Q_k\}t^n
\end{equation}
is rational, with first pole at  $t=\bigl(2\cos\bigl(\frac{\pi}{2+k}\bigr)\bigr)^{-1}$.
\end{lemma}

The rationality property remains valid for model $\cA$, see Lemma~\ref{lemma:GF_A_rat_RC}; however, due to parity issues the series \eqref{eq:def_GF_L} is zero when $\cS=\cA$ and $k$ is odd, hence a separate statement.

Proposition~\ref{prop:expression_matrix_A} may be extended immediately to models $\cB$ and $\cD$, replacing the jump matrices $A_n$ by $B_n$ and $D_n$ in $\mathcal M_{n-1,n}(\mathbb R)$, as follows:
\begin{equation}
\label{eq:def_B_D_n}
    D_n=\begin{pmatrix}
0 & 1 &0&\cdots   & 0\\
0 & 0 & 1  &\ddots&  0 \\
\vdots  & \vdots  & \ddots & \ddots  &\vdots \\
0 & \cdots & \cdots & 0 & 1
\end{pmatrix}
\quad \text{and} \quad
B_n=
\begin{pmatrix}
1 & 1 & 0 & \cdots & 0\\
0 & 1 & 1 & \ddots & \vdots\\
\vdots & \ddots & \ddots & \ddots & 0\\
0 & \cdots & 0 & 1 & 1
\end{pmatrix}.
\end{equation}
More precisely:
\begin{proposition}
\label{prop:expression_matrix_B_D}
For any $k\geq 1$, one has
\begin{align*}
\left(\sum_{n\geq0} \#_{\cB}\{(0,0)\stackrel{n}{\to} (k-i,i)\}t^n\right)_{0\leq i\leq k}& =t^kB_2 F_2 \cdots B_{k+1}F_{k+1},\\
   \left(\sum_{n\geq0} \#_{\cD}\{(0,0)\stackrel{n}{\to} (k-i,i)\}t^n\right)_{0\leq i\leq k}& =t^kD_2 F_2 \cdots D_{k+1}F_{k+1},
\end{align*}
with $B_n$, $D_{n}$ and $F_{n}$ as in \eqref{eq:def_B_D_n} and \eqref{eq:def_F_n}, respectively.
\end{proposition}

Models $\cC$ and $\cE$ contain some more intrinsic complexity, due to the following: if a walk is on the segment $S_k$, the next step could bring it either on $S_{k+1}$ or on $S_{k+2}$.

Given $k\geq 1$, we will call $\cC_k$ the set of compositions of $k$ with $1$'s and $2$'s, i.e., the set of all (ordered) $p$-tuples $(n_1,\ldots,n_p)$ such that $n_i\in\{1,2\}$ and $n_1+\cdots +n_p=k$. Though not crucial for our analysis, we observe that $\cC_k$ is a set of cardinal given by the $k$-th Fibonacci number.

Let us introduce the following notation: given $n\in\{1,2\}$ and $m\geq 1$, $J(n)_m$ denotes the following matrix:
\begin{equation*}
    J(n)_m = \left\{\begin{array}{rl}
    A_m & \text{if } n=2,\\
    B_m & \text{if } n=1 \text{ and the model is $\cC$},\\
    D_m & \text{if } n=1 \text{ and the model is $\cE$}.
    \end{array}\right.
\end{equation*}

\begin{proposition}
\label{prop:expression_matrix_C_E}
For the models $\cC$ and $\cE$, one has
\begin{multline*}
    \left(\sum_{n\geq0} \#\{(0,0)\stackrel{n}{\to} (k-i,i)\}t^n\right)_{0\leq i\leq k} =\\
    \sum_{(n_1,\ldots,n_p)\in\cC_k} t^p J(n_1)_{n_1+1}F_{n_1+1} J(n_2)_{n_1+n_2+1}F_{n_1+n_2+1}\cdots J(n_p)_{n_1+\cdots+n_p+1}F_{n_1+\cdots+n_p+1}.
\end{multline*}
\end{proposition}

\subsection{End of the proof of Theorem~\ref{thm:cont_1/2}}

The matrix-expressions of the generating functions obtained in Propositions~\ref{prop:expression_matrix_A}, \ref{prop:expression_matrix_B_D} and \ref{prop:expression_matrix_C_E}
immediately entail Proposition~\ref{prop:cont_1/2}, namely, that the series $Q(x,y,\tfrac{1}{2})$ in the variables $x$ and $y$ has finite coefficients:

\begin{proof}[Proof of Proposition~\ref{prop:cont_1/2}]
Let $\cS$ be any step set in Table~\ref{tab:list}. We can reformulate \eqref{eq:def_generating_function} as
    \begin{equation*}
    Q_\cS(x,y,t)=\sum_{k\geq 0}\sum^k_{i=0}x^{k-i}y^{i}
    \sum_{n\geq0} \#_{\cS}\{(0,0)\stackrel{n}{\to} (k-i,i)\}t^n.
    \end{equation*}
Using the expressions obtained in   Propositions~\ref{prop:expression_matrix_A}, \ref{prop:expression_matrix_B_D} and \ref{prop:expression_matrix_C_E} and Lemma \ref{lem:rationality_1D_RW}, we obtain that for any $Q_k=(k-i,i)\in S_k$, the series $\sum_{n\geq0} \#_{\cS}\{(0,0)\stackrel{n}{\to} Q_k\}t^n$ has non-negative coefficients and converges uniformly on
the closed disk $\vert t\vert\leq  \frac{1}{2}$, as indeed for any $k\geq 0$, \begin{equation*}
    \frac{1}{2}< \frac{1}{2\cos(\frac{\pi}{2+k})}.
\end{equation*}
Therefore the series
$Q_\cS (x,y,\tfrac{1}{2})$
is well defined in $\R\llbracket x,y\rrbracket$.
\end{proof}

We can actually say more:
\begin{proposition}
\label{prop:rational_coeff}
For any step set $\cS$ and any $Q_k\in S_k$, we have \begin{equation*}
    \sum_{n\geq0} \frac{\#_{\cS}\{(0,0)\stackrel{n}{\to} Q_k\}}{2^n} \in\mathbb Q.
\end{equation*}
\end{proposition}

\begin{proof}
The expressions in matrix form obtained in Propositions~\ref{prop:expression_matrix_A}, \ref{prop:expression_matrix_B_D} and \ref{prop:expression_matrix_C_E} entail that the series
$\sum_{n\geq0}\#_{\cS}\{(0,0)\stackrel{n}{\to} Q_k\} t^n$ can be computed as the coefficient of a product of matrices involving the $F_n$ in \eqref{eq:def_F_n} and other matrices with integer coefficients.
Being the inverse of a matrix whose  coefficients belong to $\mathbb Q(t)$, the coefficients of $F_n$
are in $\mathbb Q(t)$ as well. As an immediate consequence, all the coefficients of $F_n$ are rational numbers as soon as $t\in\mathbb Q$, in particular at $t=\frac{1}{2}$, as claimed in Proposition~\ref{prop:rational_coeff}.

Notice that this statement can be made more precise, using the explicit coefficients of $F_n$ computed in Proposition~\ref{prop:coeff_F_n} for general values of $t$, and their simplification at $t=\frac{1}{2}$ derived in~\eqref{eq:def_F_n_inverse_at_1/2}. In particular, the latter identity implies that for all $k\geq 1$ and all $P,Q\in S_k$, the quantity $W_k(P,Q,\frac{1}{2})\in \frac{1}{k+2}\cdot \mathbb N$ (recall from \eqref{eq:identities_matrix_inverse} that $W_k(P,Q,\frac{1}{2})$ are the coefficients of $F_{k+1}$).
\end{proof}

By Lemmas~\ref{lemma:GF_A_rat_RC} and \ref{lemma:GF_BCDE_rat_RC}, the poles $t=\bigl(2\cos\bigl(\frac{\pi}{2+k}\bigr)\bigr)^{-1}$ of the series $\sum_{n\geq0} \#_{\cS}\{(0,0)\stackrel{n}{\to} Q_k\}t^n$ have an accumulation point at $\frac{1}{2}$ as $k\to\infty$. Accordingly, one cannot evaluate the series $Q(x,y,t)$ at any point $t>\frac{1}{2}$.

\section{On the nature of the series
\texorpdfstring{$\boldsymbol{Q(x,y,t)}$}{Q(x,y,t)} for \texorpdfstring{$\boldsymbol{t\in\bigl(0,\frac 12\bigr]}$}{t in (0,1/2)}}
\label{sec:main}

In this section,
for any fixed step set $\mathcal{S}$ in Table~\ref{tab:list},
we explore the nature of the ordinary generating function
$Q(x,y)\equiv Q_{\cS}(x,y,t)$, for all the values of $t\in\bigl(0,\frac{1}{2}\bigr]$.
As we proved in the previous section, it makes sense to evaluate $t$ from $0$ to $\frac{1}{2}$, but one cannot go further on the real line.
Recall that to simplify the notation we will omit writing the dependence in $\cS$ and $t$, when the context is clear.
\par
We are going to show that the series $Q(x,y)$ is differentially transcendental over $\C(x,y)$
(and hence over $\Q(x,y)$)
for $t\in\bigl(0,\frac 12\bigr)$ and that it is strongly differentially transcendental for $t=\frac 12$, in the sense that it is differentially transcendental over the germs of bivariate meromorphic functions at $0$, as formalized in the following definition:

\begin{definition}\label{def:sdt}
Let $f\in \C[[x]]$ (resp.\ $\C[[x,y]]$). We say that $f$ is \emph{differentially algebraic} over  $\C(x)$ (resp.\ $\C(x,y)$) if there exists an integer $n$ and a non-zero polynomial~$P$
with coefficients in $\C(x)$ (resp.\ $\C(x,y)$) in $n+2$ (resp.\ $(n+1)^2+2$) variables such that:
\[
    P\l(x,f,\dots,\frac{d^nf}{dx^n}\r)=0
    \quad \left(\text{resp.} \quad
    P\l(x,y,f,\frac{df}{dx},\frac{df}{dy},\dots,\frac{d^{2n}f}{dy^ndx^n}\r)=0.\right)
    \]
We say that $f$ is differentially transcendental over $\C(x)$ (resp.\ $\C(x,y)$) if it is not differentially algebraic.
\par
{\rm Mutatis mutandis}, one defines \emph{strong differential transcendence}, i.e., the differential transcendence of a formal power series over the field of germs of meromorphic functions $\C(\{x\})$
(resp.\ $\C(\{x,y\})$), that is the field of functions that are meromorphic in an open, non-specified neighborhood of $0$.
\end{definition}

The main result of this section is:

\begin{theorem}\label{thm:main-Galois}
For any model in the list $\mathcal{A}$, $\mathcal{B}$, $\mathcal{C}$, $\mathcal{D}$, $\mathcal{E}$ and
for any $t\in\bigl(0,\frac 1 2\bigr]$, the formal power series
$Q(x,0)$ (resp.\ $Q(0,y)$, $Q(x,y)$)\
is differentially transcendental over $\C(x)$
(resp.\ $\C(y)$, $\C(x,y)$).
For $t=\frac 12$, the power series $Q(x,0)$ (resp.\ $Q(0,y)$, $Q(x,y)$)
is strongly differentially transcendental.
\end{theorem}

\begin{remark}
\normalfont
The results of Theorem~\ref{thm:main-Galois} extend those of
 \cite{DHRS0}, where the authors consider only \emph{transcendental} values of $t\in\bigl(0,\frac{1}{\# \mathcal S}\bigr)$. Note that they consider models with weights, whereas we do not. On the other hand, we consider strong differential transcendence, while~\cite{DHRS0} deals with differential transcendence.
\end{remark}

As pointed out in the introduction, the models $\mathcal{A}$, $\mathcal{B}$, $\mathcal{C}$, $\mathcal{D}$, $\mathcal{E}$ are called genus-zero models because their kernel \eqref{eq:kernel}, namely
\[
    K_{\cS}(x,y,t) := xy\bigl(1-t \cdot
    {\sum_{(i,j)\in\cS}x^{i}y^{j}}\bigr),
\]
defines a genus-zero curve. This is equivalent to saying that the curve $K_{\cS}(x,y,t)$ admits a rational parametrization. The choice of parametrization is an important point in our proof. We will return to this point in \S\ref{subsec:remarks}, where we comment on and compare our choices with previous ones in the literature.

In the proof below, it will be convenient to set $t=\frac{v}{1+v^2}$, with $v\in (0,1]$; this substitution has proven useful in other related contexts, see e.g.\ \cite[Sec.~2.4]{MiRe09}.

\subsection{General strategy for the proof and the proof for model $\cA$}

We will treat the five singular models one by one, using difference Galois theory.
At the end of the section, after the proof, we will comment on the fact that for $t\in\bigl(0,\frac 1 2\bigr)$ one cannot expect to find strongly differential transcendental generating functions.

\begin{proof}[Strategy of the proof of Theorem~\ref{thm:main-Galois}]
The peculiarity of the five models of Table~\ref{tab:list} is that the kernel curve $K(x,y,t)=0$ has genus zero for any $t\in\bigl(0,\frac{1}{2}\bigr]$ (except in the case of model $\mathcal A$ at $t=\frac12$, where the kernel polynomial becomes reducible), which is equivalent to the existence of a rational parametrization.  We will consider the cases $t\in\bigl(0,\frac{1}{2}\bigr)$ and $t=\frac{1}{2}$ separately, since we did not find a parametrization with a non-trivial specialization for any value of~$t$.
We set $t=\frac{v}{1+v^2}$, with $v\in (0,1]$.

For $v\in(0,1)$ we will find $x_0(s),\widetilde{x}_0(s),y_0(s)\in\C(s)$ such that
$(x_0(s),y_0(s))$  and $(\widetilde{x}_0(s),y_0(s))$ are both parametrizations of $K(x,y,v/(v^2+1))=0$ and $\widetilde{x}_0(s)=x_0(v^2\,s)$.
Plugging the two parametrizations in the functional equation~\eqref{eq:kernel-funct-eq}
and taking the difference of the two expressions, we obtain
    \begin{equation}\label{eq:functional-equation}
    \frac{v}{v^2+1} \cdot \left( x_0^2 Q(x_0,0) - \widetilde{x}_{0}^2  Q(\widetilde{x}_{0},0) \right)
		= (x_0-\widetilde{x}_{0}) y_0.
    \end{equation}
Therefore, as an element of $\C[[s]]$, the formal power series
\begin{equation}
    \label{eq:def_G0}
    G_0(s):=\frac{v}{v^2+1} \cdot (x_0(s))^2 \cdot Q(x_0(s),0)
\end{equation}
satisfies a $v^2$-difference equation of the form $G_0(v^2\,s)-G_0(s)=R(s)\in\C(s)$.
We prove that such an equation does not have any rational solution and conclude that $G_0(s)$ is differentially transcendental over $\C(s)$, thanks to the following result:

\begin{theorem}[{\cite{Ishizaki-hypertransc}, \cite[Thm.~2]{Ogawara-formalIshizaki}}]
\label{thm:Ogawara}
Let $q\neq 0$ be a complex number, not a root of unity.
If $f\in\C(\!(s)\!)\setminus \C(s)$ satisfies the functional equation
    \[
    f(qs)=Q(s)f(s)+R(s),\quad\hbox{for some $Q(s),R(s)\in\C(s)$, $Q(s)\neq 0$},
    \]
then
$f$ is differentially transcendental over $\C(s)$.
\end{theorem}

Since $x_0(s)\in\C(s)$ for any $v\in (0,1)$, $G_0(s)$ is  differentially transcendental over $\C(s)$ if and only if
$Q(x,0)$ is  differentially transcendental over $\C(x)$.
Moreover, from \eqref{eq:kernel-funct-eq},
the differential transcendence of $Q(x,0)$ over $\C(x)$ and that of $Q(0,y)$ over $\C(y)$ are equivalent, since:
    \[
    x_0(s)y_0(s)-tx_0(s)^2Q(x_0(s),0)=ty_0(s)^2Q(0,y_0(s))\,.
    \]
Finally, \eqref{eq:kernel-funct-eq} allows us to conclude that $Q(x,y)$ is differentially transcendental over $\C(x,y)$, using that $x$ and $y$ are algebraically independent.

Let us now turn to the case $v=1$, for which we shall reason in the same way. We find $x_1(s),\widetilde{x}_1(s), y_1(s)\in\C(s)$ such that
$(x_1(s),y_1(s))$  and $(\widetilde{x}_1(s),y_1(s))$ are both parametrizations of $K(x,y,\frac{1}{2})=0$ and $\widetilde{x}_1(s)=x_1(\tau(s))$,
where $\tau$ is a homography with a single fixed point.
Therefore, the formal power series
\begin{equation}
    \label{eq:G1}
    G_1(s):=\frac{1}{2} \cdot (x_1(s))^2 \cdot Q\bigl(x_1(s),0,\tfrac1 2\bigr)
\end{equation}
satisfies a %
difference equation of the form $G_1(\tau(s))-G_1(s)=R(s)\in\C(s)$.
We prove that such an equation does not have rational solutions and conclude, as for $v\in(0,1)$,
thanks to the following result.

\begin{theorem}[{\cite[Thm.~3, Thm.~40 and Cor.~41]{bostan-divizio-raschel-BELL}}]
\label{thm:OurTheorem;)}
Let $\tau$ be an automorphism of $\C(s)$, having $0$ as unique
fixed point. If $G(s)\in\C((s))\setminus\C(s)$ satisfies a functional
equation of the form
    \[\tau(G(s))=Q(s)G(s)+R(s),
    \quad\hbox{for some $Q(s),R(s)\in\C(s)$, $Q(s)\neq 0$,}\]
then $G(s)$ is differentially transcendental over the field
$\C(\{s\})$
of germs of meromorphic functions at $0$.
\end{theorem}
As in the previous case, one can recover the information on $Q(x,0)$, $Q(0,y)$ and $Q(x,y)$ from
the strongly differential transcendence of $G(s)$.
\end{proof}

We will now detail the proof of the above theorem for the five genus-zero models, one by one.
We start with model $\mathcal{A}$ which is a baby example, compared to the others.

 \begin{proof}[Proof of Theorem~\ref{thm:main-Galois} in the case of model \texorpdfstring{$\cA$}{A}]
In this case, for $v\in(0,1)$, we have:
    \begin{align}
        K(x,y) &=xy-t \left({y}^{2}+{x}^{2}{y}^{2}+{x}^{2}\right);\nonumber\\
        \displaystyle x_0(s) &=\frac{\left( 1-v^2 \right)s}{ v\left( {s}^{2}+1 \right)}\quad\text{and}\quad
            y_0(s) = {\frac { \left( 1-v^2 \right)   s}{{v}^{2}{s}^{2}+1}};
            \label{eq:parametrization_A}\\
        \displaystyle  \widetilde{x}_{0}(s) &= {\frac { \left( 1 - {v}^{2} \right) v \, s}{ {v}^{4}s^2 + 1 }},\quad\hbox{ with $\wtilde x_0(s)=x_0(v^2 s)$.}\nonumber
    \end{align}
Therefore, the formal power series $G_0(s)$
in \eqref{eq:def_G0} satisfies the $v^2$-difference equation:
    \begin{align*}
        G_0(s)-G_0(v^2s) & =
        \displaystyle\frac{({v}^{2}-1)^3}{v}
        {\frac {{s}^{2} (v^2s^2-1)}
        {\left( {s}^{2}+1 \right)  \left( {v}^{4}{s}^{2}+1 \right)
        \left( {v}^{2}{s}^{2}+1 \right) }}\\
         & =
        \displaystyle\frac{({v}^{2}-1)}{v}
        \l({\frac {1}{{s}^{2}+1}}-{\frac{2}{{v}^{2}{s}^{2}+1}}
         +{\frac{1}{{v}^{4}{s}^{2}+1}}\r).
    \end{align*}
By Theorems~\ref{thm:Ogawara} and \ref{thm:OurTheorem;)}, we only need to prove that, for any $v\in (0,1)$,
the power series
$G_0(s)$ is not rational.
We set
$\widetilde G_0(s)=\frac{v}{2(v^2-1)}G_0(s)-\frac{1}{2(s^2+1)}$, so that
    \begin{equation}
    \label{eq:q-diff-ModelA-bis}
        \widetilde G_0(s)-\widetilde G_0(v^2s)=
        \frac{1}{v^4{s}^{2}+1}-\frac{1}{{v}^{2}{s}^{2}+1}.
    \end{equation}
Of course, $G_0(s)$ is rational if and only if $\widetilde{G}_{0}(s)$ is rational.
Let us suppose by contradiction that~\eqref{eq:q-diff-ModelA-bis} has a rational solution $\widetilde G_0\in\C(s)$.
If $\widetilde G_0\in\C(s)$ then the left-hand side of \eqref{eq:q-diff-ModelA-bis} would have
at least two poles whose quotient is in $v^{2\Z}$.
The poles of the right-hand side of \eqref{eq:q-diff-ModelA-bis}
are $\pm i/v^2,\pm i/v$.
Since $v\in(0,1)$ is real and positive no quotient of two of these poles is in $v^{2\Z}$.
We have obtained a contradiction and therefore we can conclude
that $G_0(s)$ is differentially transcendental over $\C(s)$.

Now we consider the case $t=\frac{1}{2}$, or equivalently the case $v=1$.
For such value of $v$ the parametrization \eqref{eq:parametrization_A} degenerates,
however the kernel curve becomes reducible:
\begin{equation}\label{eq:modAv=1}
        xy - \frac12 \left(x^2+y^2+x^2y^2\right)=
        \frac12 \left( ix-iy+xy \right)  \left( ix-iy-xy \right).
\end{equation}
The first factor provides the parametrization $\l(x,y(x):=\frac{x}{1+ix}\r)$ of the kernel curve.
Notice that the diagonal symmetry of model $\cA$ implies that $Q(x,0) = Q(0,x)$, therefore
we deduce from~\eqref{eq:kernel-funct-eq}
that $G_1(x)$
in \eqref{eq:G1} satisfies the functional equation:
    \begin{equation}\label{eq:ModA1/2}
    G_1\l(\tau(x)\r)=-G_1(x)+\frac{x^2}{1+ix},\quad\hbox{with } \tau(x)=\frac{x}{1+ix}.
    \end{equation}
As above, we need to prove that $G_1(x)$ cannot be rational.
The inhomogeneous term $\frac{x^2}{1+ix}$ in \eqref{eq:ModA1/2} has a single pole at $x=i$,
which is not a fixed point for $x\mapsto \frac{x}{1+ix}$. As a consequence,
if $G_1(x)$ was rational, $G_1\l(\frac{x}{1+ix}\r)+G_1(x)$ would have at least two poles in the $\tau$-orbit of $i$.
We have obtained a contradiction, hence $G_1(x)$ is not rational.
\end{proof}

In the appendix we give detailed proofs for the remaining four models.
For the reader convenience we present a table with pointers to the formulas in the proofs: equations are too long to be recall here individually.

\bigskip
\begin{tabular}{|c|c|c|c|c|c|}
\hline
\textbf{model} & \textbf{step set} &\multicolumn{2}{c|}{\textbf{ $v\in(0,1)$}} & \multicolumn{2}{c|}{\textbf{$v=1$}} \\
 \cline{3-6}
 && parametrization & functional eq. & parametrization & functional eq. \\
\hline
$\cA$ & $\textcolor{red}{\2stepset00100101}$
    & \eqref{eq:parametrization_A} & \eqref{eq:q-diff-ModelA-bis} & \eqref{eq:modAv=1} & \eqref{eq:ModA1/2} \\
$\cB$ & $\textcolor{red}{\2stepset00101110}$
        & \eqref{eq:ParaBdhrs} & \eqref{eq:feqBv<1} & \eqref{eq:paraBv=1} & \eqref{eq:feqBv=1} \\
$\cC$ & $\textcolor{red}{\2stepset00101111}$
        & \eqref{eq:ParaCmaple} & \eqref{eq:feqCv=1} & \eqref{eq:paraCv=1} & \eqref{eq:feqCv<1} \\
$\cD$ & $\textcolor{red}{\2stepset00100110}$
        &  \eqref{eq:ParaDdhrs} & \eqref{eq:feqDv<1} & \eqref{eq:paraDv=1} & \eqref{eq:feqDv=1} \\
$\cE$ & $\textcolor{red}{\2stepset00100111}$
        & \eqref{eq:ParaEmaple} & \eqref{eq:feqEv<1} & \eqref{eq:paraEv=1} & \eqref{eq:feqEv=1} \\
\hline
\end{tabular}

\bigskip

\section{Remarks on the statement and the proof of the main theorem}
\label{subsec:remarks}

\paragraph{About the different kind of transcendence.}
We would like to discuss the dichotomy in the behavior of $Q(x,y)$ in Theorem~\ref{thm:main-Galois}, with respect to the values of $t$.
First of all, we notice that we cannot expect solutions of $q$-difference equations to be strongly differentially transcendental;
in fact, solutions of $q$-difference equations can be (and one could even say tend to be)\ meromorphic in a neighborhood of $0$, as the next example shows.

\begin{exa}
For $q\in\C$, $|q|>1$, the $q$-exponential series $$e_q(t):=\sum_{n\geq 0}\frac{t^n}{\prod_{i=1}^n(q^i-1)}$$ is an entire function,
solution of the functional equation $y(qt)=(1+t)y(t)$.
One derives from the equation itself that it has a half spiral of simple zeros.
Taking the logarithmic derivative with respect to $t\frac{d}{dt}$,
one proves that the meromorphic function $\epsilon_q(t):=t\frac{e_q(t)'}{e_q(t)}$ over $\C$ is solution of:
    \[
    y(qt)=y(t)+\frac{t}{1+t}\,.
    \]
\end{exa}

As we have shown in the proofs above, for $t\in\bigl(0,\frac 12\bigr)$ we find a $q$-difference equation for the generating function (after replacing the variable by the uniformizing variable),
i.e., a functional equation with respect to a homography with two fixed points. For $t=\frac{1}{2}$, the two fixed points coincide, and
we get a functional equation with respect to a homography with one fixed point. One can heuristically argue that the confluence of two special points to one causes an increase in the complexity of the solutions, i.e., the fact of jumping from a usual differential transcendence to a strong differential transcendence. In the next section we give a probabilistic interpretation of $Q(x,y)$ for any $t\in\bigl(0,\frac{1}{2}\bigr]$, which makes the point $\frac{1}{2}$ appear to be critical from a probabilistic point of view as well. This corroborates the change in behavior at this point.

\paragraph{Choice of the parametrization.}
The idea of considering a parametrization as an important tool to solve probabilistic or combinatorial functional equations goes back to \cite{FIM17}.
For $v\in(0,1)$, the parametrization \eqref{eq:parametrization_A} of the kernel curve of model $\mathcal{A}$ and
the parametrization \eqref{eq:ParaDdhrs} of model $\mathcal{D}$ coincide with the ones used in \cite{DHRS0}.
In the case of model $\mathcal{B}$, one can recover the parametrization used in \cite{DHRS0} from parametrization \eqref{eq:ParaBdhrs}
applying the variable change $s\mapsto\frac{s}{\sqrt{v}}$. The variable change $s\mapsto\frac{s}{\sqrt{v^2 - v + 1}}$
allows one to recover the parametrization in \cite{DHRS0} from \eqref{eq:ParaCmaple}, while the variable change
$s\mapsto\frac{s}{\sqrt{v^4 - v^2 + 1}}$ works for \eqref{eq:ParaEmaple}.
So, up to minor variable changes, for $t\in(0,1/2)$, we essentially use the same parametrization used in \cite{DHRS0}.
This is not surprising since all the parametrizations are proper, therefore equivalent up to a homographic change of the parameter $s$.
Thanks to this property, all the parametrizations have been found alternating the use of Maple packages and homographic changes
of variable, found by, so to say, human intuition, to simplify the final functional equations.
\par
The situation is more articulated for $v=1$. As far as the models $\mathcal{A}$, $\mathcal{C}$ and $\mathcal{E}$ are concerned,
we use a specific parametrization valid only for $v=1$.
For models $\mathcal{B}$
and $\mathcal{D}$, a variable change allows us to obtain a parametrization that specialize properly at $v=1$.
The parametrization \eqref{eq:ParaBdhrs} of model~$\mathcal{B}$ becomes trivial for $v=1$.
However, we can consider another parametrization that has a ``good'' specialization at $v=1$,
obtained applying a variable change:
    \begin{equation}
    \label{eq:ParaBnotes}
    x_0\l(\frac{s}{s+1-v^2}\r)=-{\frac {s \left(s-{v}^{2}+1 \right) }{v \left( s+v+1 \right) }}\quad \text{and}\quad
    y_0\l(\frac{s}{s+1-v^2}\r)=-{\frac {s \left(s-{v}^{2}+1 \right) }{ \left( s+1 \right)  \left( s+v+1 \right) }}.
    \end{equation}
Moreover we have:
    \[
    \widetilde x_0\l(\frac{s}{s+1-v^2}\r)=
   -{\frac {s \left(s-{v}^{2}+1 \right) v}{ \left( s+1 \right)  \left(
 \left( {v}^{2}+v+1 \right) s+v+1 \right) }},
    \]
so that
    \[
    \widetilde x_0\l(\frac{s}{s+1-v^2}\r)=x_0\l(\frac{s}{s+1-v^2}\circ{\frac {v^2s}{s+1}}\r).
    \]
Also in the case of model $\mathcal{D}$ the parametrization \eqref{eq:ParaDdhrs} becomes trivial for $v=1$.
However, we can consider another parametrization that has a ``good'' specialization at $v=1$,
obtained applying a variable change:
    \begin{equation}\label{eq:ParaDnotes}
    x_0\l(\frac{s}{v-1-vs}\r)=-\frac{\left(v +1\right)^{2} s \left(1+\left(s -1\right) v \right)}{\left(s -1\right)^{2} v^{2}}\ \text{and}\
    y_0\l(\frac{s}{v-1-vs}\r)=\frac{\left(v +1\right)^{2} s \left(1+\left(s -1\right) v \right)}{\left(s v +1\right) \left(s -1\right) v}.
    \end{equation}
Moreover we have:
    \[
    \widetilde x_0\l(\frac{s}{v-1-vs}\r)=
    -\frac{\left(v +1\right)^{2} s \left(s v -v +1\right)}{\left(s v +1\right)^{2}}\,,
    \]
so that
    \[
    \widetilde x_0\l(\frac{s}{v-1-vs}\r)=
    x_0\l(\frac{s}{v-1-vs}\circ{\frac{s \,v^{2}}{s\,v(v+1) +1}}\r).
    \]
Specializing $v=1$ at \eqref{eq:ParaBnotes} and \eqref{eq:ParaDnotes}, one finds \eqref{eq:paraBv=1} and \eqref{eq:paraDv=1}, respectively.

\section{Probabilistic interpretation using Green functions}
\label{sec:probab}

As explained in our introduction, one of the main messages of the present work is that for all the models listed in Table~\ref{tab:list}, the point $t=\tfrac{1}{2}$ is very special for the generating function $Q(x,y,t)$, in particular because the differential transcendence degenerates to a strong differential transcendence.
In this section, we give two probabilistic interpretations that single out the critical role of $t=\tfrac{1}{2}$. Both are based on the representation of the generating function in terms of a classical Green function:
\begin{equation}
\label{eq:relation_Q_G_0}
    Q(x,y,t) = \sum_{i,j \geq 0} G(i,j \vert \#\mathcal{S}\, t)\, x^i y^j,
\end{equation}
where the notation will be specified below. Since $G(i,j \vert \#\mathcal{S}\, t)$ is the Green function of a random walk on a finite graph, it is a rational function of $t$.

In Section~\ref{subsec:Green_function_segments}, we exploit the combinatorial decomposition of singular walks into walks on segments to obtain explicit expressions for the above Green function as finite sums of products of Green functions associated with killed random walks on segments. Within this framework, $t=\tfrac{1}{2}$ emerges as the maximal radius of a disc centered at the origin beyond which poles occur.

In Section~\ref{subsec:Cramer}, we introduce a Cramér transform, which allows us to reinterpret \eqref{eq:relation_Q_G_0} in terms of the more natural Green function $G(i,j \vert 1)$; this specific Green function corresponds to the expected number of visits of the random walk. While this transformation is well defined for all $t \in (0,\tfrac{1}{2})$, it becomes singular at $t=\tfrac{1}{2}$.

Finally, we provide a third interpretation in terms of branching random walks (Section~\ref{subsec:branching}), and investigate the effect of allowing weighted jumps in connection with the Cramér transform (Section~\ref{subsec:weights}).

\subsection{Green function for random walks on segments}
\label{subsec:Green_function_segments}

Let $(X(n))_{n\geq 0}$ be a classical simple random walk on $\mathbb Z$, with uniform jump probabilities to the left and right nearest neighbors. Let $\tau_k$ define the first exit time of the random walk from the segment $\widetilde S_k=\{-k,\ldots,k\}$, in bijection with the set $S_k$ introduced in \eqref{eq:segments}, i.e.,
\begin{equation*}
    \tau_k = \inf\{ n\geq 0 : X(n)\notin \widetilde S_k\}.
\end{equation*}
By definition (see e.g.\ \cite[Chap.~1]{Wo00}), the Green function for the random walk killed when exiting the segment $\widetilde S_k$ is
\begin{equation}
\label{eq:def_Green_1D}
    G_k(P,Q \vert t) = \sum_{n\geq 0} \mathbb P_P\bigl(X(n) = Q,\tau_k>n\bigr) t^n,
\end{equation}
where $P\in \widetilde S_k$ and $Q\in \widetilde S_k$ are arbitrary starting and ending points, respectively, and under $\mathbb P_P$ the random walk $(X(n))_{n\geq 0}$ starts at $P$ almost surely, i.e., $X(0)=P$.

It is well known (see again \cite[Chap.~1]{Wo00}) that the radius of convergence of $G_k(P,Q\vert t)$ in \eqref{eq:def_Green_1D} as a power series in $t$ is independent of $P$ and $Q$; it is denoted by $\frac{1}{\rho}$, where $\rho\in(0,1]$ is often called the \emph{spectral radius} of the random walk.
In the particular case $t=1$, the Green function \eqref{eq:def_Green_1D} admits the following classical interpretation in terms of a mean number of visits before exiting $\widetilde S_k$:
\begin{equation}
\label{eq:interpretation_visits_dim-1}
    G_k(P,Q\vert 1) = \mathbb E_{P}\left(\sum_{n\geq 0} \mathds{1}_{\{X(n) = Q,\tau_k>n\}}\right),
\end{equation}
where $\mathds{1}$ denotes the indicator function. Using the above notation, we immediately deduce the following reformulation of the series $W_k(P,Q,t)$ introduced in \eqref{eq:GF_segments}:
\begin{lemma}
\label{lem:series-to-Green}
Let $k\geq 1$. For any $P,Q\in S_k$, let $\widetilde P,\widetilde Q$ be the corresponding points on $\widetilde S_k$. For any $\vert t\vert <\frac{1}{2\cos(\frac{\pi}{2+k})}$, we have
\begin{equation*}
    W_k(P,Q,t) = G_k(\widetilde P,\widetilde Q\vert 2t).
\end{equation*}
\end{lemma}

\begin{remark}
\label{rem:explanation}
\normalfont The point $\frac{1}{2}$ appears as the smallest upper bound of all radii of convergence of the series in Lemma~\ref{lem:series-to-Green}.
\end{remark}

We further notice that the coefficient $2$ in the Green function $G_k$ simply comes from normalising the transition probabilities of the random walk $(X(n))_{n\geq 0}$ to $1$, while in the enumerative series $W_k$ the weights add up to $2$. In particular, $W_k(P,Q,\frac{1}{2}) = G_k(\widetilde P,\widetilde Q\vert 1)$. With this perspective, the case $t=\frac{1}{2}$ is special as it corresponds exactly to the Green function of the  random walk with $t=1$.

Finally, together with the identity \eqref{eq:decomp_GF}, which expresses the series $\sum_{n\geq0} \#_{\cA}\{(0,0)\stackrel{n}{\to} Q_{2k}\}t^{n}$ as a finite sum of terms $W_k(P,Q,t)$, we deduce that $\sum_{n\geq0} \frac{\#_{\cA}\{(0,0)\stackrel{n}{\to} Q_{2k}\}}{2^{n}}$ is a finite sum of products of Green functions as above.

\subsection{An interpretation using Cram\'er's transform}
\label{subsec:Cramer}

In this paragraph, we provide a concrete interpretation of the series $Q(x,y,\frac{1}{2})$, using Green functions and the idea of Cram\'er's transform. Here we will work directly with the two-dimensional model, without using the decomposition with walks on segments $S_k$ as in Figure~\ref{fig:recursive_cons}.

Given a step set $\cS$, it is convenient to introduce the associated random walk $(Z(n))_{n\geq 0}$ with uniform probability distribution on $\cS$, killed when leaving the positive quarter plane. By definition, and similarly to \eqref{eq:def_Green_1D}, the Green function at $(i,j)$ is
\begin{equation}
\label{eq:def_Green}
    G(i,j\vert t) = \sum_{n\geq 0} \mathbb P_{(0,0)}\bigl(Z(n) = (i,j),\tau>n\bigr)t^n,
\end{equation}
where $\tau$ is the first time the walk goes out of the quadrant.

The most natural value of $t$ in \eqref{eq:def_Green} is $1$, because as in \eqref{eq:interpretation_visits_dim-1} in this particular case we have an interpretation of the Green function as the expected number of visits of the random walk at the given site $(i,j)$:
\begin{equation*}
    G(i,j\vert 1) = \mathbb E_{(0,0)}\left(\sum_{n\geq 0} \mathds{1}_{\{Z(n) = (i,j),\tau>n\}}\right).
\end{equation*}
Using the obvious identity (see for instance \cite[Eq.~(1)]{BoRaSa14})
\begin{equation*}
    \mathbb P_{(0,0)}\bigl(Z(n) = (i,j),\tau>n\bigr) = \frac{\#_\cS\{(0,0)\stackrel{n}{\to} (i,j)\}}{\#\cS ^n},
\end{equation*}
we deduce a direct relation between Green functions and the generating function $Q(x,y,t)$, namely \eqref{eq:relation_Q_G_0}.
We can first observe that evaluating \eqref{eq:relation_Q_G_0} at $t=\frac{1}{\#\cS }$ gives
\begin{equation*}
    Q(x,y,\tfrac{1}{\#\cS }) = \sum_{i,j\geq 0}  G(i,j\vert1)x^iy^j.
\end{equation*}
Is it possible to evaluate \eqref{eq:relation_Q_G_0} at a value of $t>\frac{1}{\#\cS }$? Ultimately at  $t=\frac{1}{2}$? (And thus the Green function at $\frac{\#\cS }{2}$.)

To answer these questions, we will use Cram\'er transform, a classical tool in random walk theory, recalled in \cite[Sec.~2.3]{BoRaSa14} in our combinatorial context. Introduce a new random walk, say $(W(n))_{n\geq0}$, whose step set $ \cS$ is the same as that of $(Z(n))_{n\geq0}$, but whose distribution is no longer uniform: instead the probability of moving in direction $(i,j)$ is given by
\begin{equation*}
    \frac{\alpha^i\beta^j}{\chi_\cS(\alpha,\beta)},
\end{equation*}
for some positive quantities $\alpha,\beta$ (simply take $\alpha=\beta=1$ to get the uniform distribution)\ and $\chi_\mathcal S$ defined in \eqref{eq:def_inventory}.

The local probabilities of the random walks $(Z(n))_{n\geq0}$ and $(W(n))_{n\geq0}$ are related as follows (see again \cite{BoRaSa14})
\begin{equation*}
    \mathbb P_{(0,0)}\bigl(W(n) = (i,j),\tau>n\bigr) = \alpha^i \beta^j \left(\frac{\#\cS }{\chi_\cS(\alpha,\beta)}\right)^n\mathbb P_{(0,0)}\bigl(Z(n) = (i,j),\tau>n\bigr),
\end{equation*}
and thus, denoting $G_{(\alpha,\beta)}$ the Green function for the random walk $(W(n))_{n\geq0}$, one has
\begin{equation}
\label{eq:Green-to-Green}
    G(i,j\vert t) = \alpha^{-i}\beta^{-j} G_{(\alpha,\beta)}(i,j\vert \tfrac{\chi_\cS(\alpha,\beta)}{\#\cS }t).
\end{equation}
The above identity allows us to interpret the Green function $G(i,j\vert t)$ for a general value of $t$ in terms of a classical Green function at $t=1$, provided in the right-hand side of \eqref{eq:Green-to-Green} it is possible to adjust the parameters so as to have $\tfrac{\chi_\cS(\alpha,\beta)}{\#\cS }t=1$.

For our purposes, we want to evaluate \eqref{eq:Green-to-Green} at $t=\frac{\#\cS }{2}$, because indeed
\begin{equation}
\label{eq:QGG}
    [x^iy^j]Q(x,y,\tfrac{1}{2}) =  G(i,j\vert \tfrac{\#\cS }{2}) = \alpha^{-i}\beta^{-j} G_{(\alpha,\beta)}(i,j\vert \tfrac{\chi_\cS(\alpha,\beta)}{2}).
\end{equation}
Unfortunately, for any given singular model, we cannot find values of the parameters $\alpha,\beta$ such that $\frac{\chi_\cS(\alpha,\beta)}{2}=1$. Indeed, since $\cS$ contains the jumps $(-1,1)$ and $(1,-1)$, we have
\begin{equation*}
    \chi_\cS(\alpha,\beta)>\frac{\alpha}{\beta}+\frac{\beta}{\alpha}\geq 2.
\end{equation*}
As a conclusion, it is possible to interpret $[x^iy^j]Q(x,y,t)$ as a classical Green function for any $t\in\bigl[\frac{1}{\#\cS},\frac{1}{2}\bigr)$, but not at the limiting point $t=\frac{1}{2}$. However, we can take $\alpha=\beta=\varepsilon\to 0$, and in this case $\chi_\cS(\alpha,\beta)\to2$. So a possible interpretation of $[x^iy^j]Q(x,y,\tfrac{1}{2})$ is the following:
\begin{equation*}
    [x^iy^j]Q(x,y,\tfrac{1}{2}) = \lim_{\varepsilon\to 0}  \frac{G_{(\varepsilon,\varepsilon)}(i,j\vert \tfrac{\chi_\cS(\varepsilon,\varepsilon)}{2})}{\varepsilon^{i+j}}.
\end{equation*}

\subsection{Branching random walks}\label{subsec:branching}
To give an interpretation of $[x^iy^j]Q(x,y,\tfrac{1}{2})$ as the Green function of a concrete probabilistic model, we will use branching random walks. While this model is very classical in probability theory, see e.g.\ \cite{Shi}, it has been less studied from a combinatorial perspective. See \cite{nondeterministic} for a connection with non-deterministic paths.

\begin{figure}
    \centering
\begin{tikzpicture}[scale=0.33]
    \draw[->,black,very thick] (0,0) -- (0,13);
    \draw[->,black,very thick] (0,0) -- (13,0);
    \draw[->,purple,very thick] (0,0) -- (1.5,1.5);
    \draw[->,violet,very thick] (1.5,1.5) -- (0,3);
    \draw[->,violet,very thick] (1.5,1.5) -- (3,3);
    \draw[->,blue,very thick] (0.1,3) -- (1.6,1.5);
    \draw[->,blue,very thick] (3,3) -- (4.5,1.5);
    \draw[->,blue,very thick] (3,3) -- (1.5,4.5);
    \draw[->,green,very thick] (1.5,4.5) -- (0,6);
    \draw[->,green,very thick] (1.5,1.5) -- (3,0);
    \draw[->,green,very thick] (4.5,1.5) -- (6,3);
    \draw[->,yellow,very thick] (0,6) -- (1.5,7.5);
    \draw[->,yellow,very thick] (6,3) -- (7.5,4.5);
    \draw[->,yellow,very thick] (3,0) -- (4.5,1.5);
    \draw[->,orange,very thick] (4.5,1.5) -- (6,0);
    \draw[->,orange,very thick] (4.6,1.5) -- (3.1,3);
    \draw[->,orange,very thick] (7.5,4.5) -- (9,6);
    \draw[->,orange,very thick] (7.5,4.5) -- (6,6);
    \draw[->,orange,very thick] (1.5,7.5) -- (0,9);
    \draw[->,orange,very thick] (1.5,7.5) -- (3,6);

     \draw[->,red,very thick] (6,0) -- (7.5,1.5);
    \draw[->,red,very thick] (3,3) -- (4.5,4.5);
    \draw[->,red,very thick] (9,6) -- (7.5,7.5);
    \draw[->,red,very thick] (6,6) -- (4.5,7.5);
    \draw[->,red,very thick] (0,9) -- (1.5,10.5);
    \draw[->,red,very thick] (3,6) -- (4.5,4.5);

    \draw[->,magenta,very thick] (7.5,1.5) -- (9,0);
    \draw[->,magenta,very thick] (7.5,1.5) -- (9,3);
    \draw[->,magenta,very thick] (4.5,4.5) -- (6,3);
    \draw[->,magenta,very thick] (4.6,4.5) -- (6.1,3);
    \draw[->,magenta,very thick] (7.5,7.5) -- (9,6);
    \draw[->,magenta,very thick] (7.5,7.5) -- (9,9);
    \draw[->,magenta,very thick] (4.5,7.5) -- (6,9);
    \draw[->,magenta,very thick] (1.5,10.5) -- (0,12);
    \draw[->,magenta,very thick] (1.5,10.5) -- (3,12);
    \draw[->,magenta,very thick] (4.5,4.5) -- (6,6);
   \end{tikzpicture}
   \begin{quote}
    \caption{A path of branching random walk. Each particle jumps according to $\cA$ and then splits into a random number of particles, uniformly distributed in $\{1,2\}$. Individuals exiting the quarter plane are killed. }
    \label{fig:rainbow}
    \end{quote}
\end{figure}
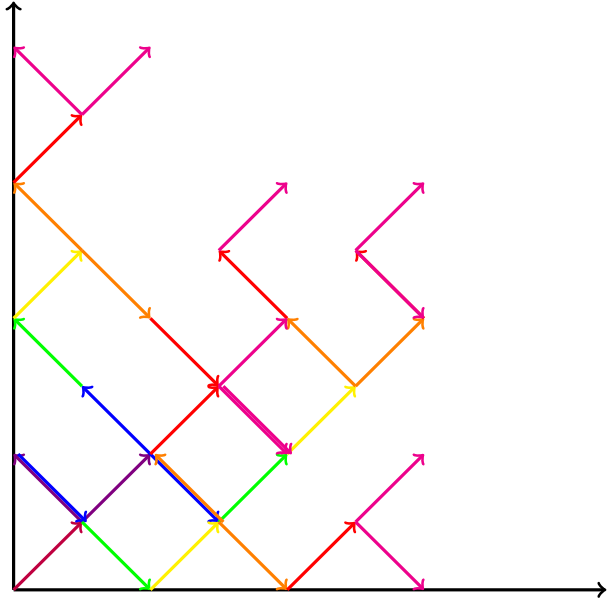

Let us briefly describe the model. We introduce an offspring distribution $\nu$ on $\{0,1,2,\ldots \}$. The first particle (called the \emph{ancestor})\ starts at the origin and jumps (according to one given model $\cS$, for instance $\cA$); then it immediately splits into a random number of particles, distributed according to $\nu$. Then each individual of the new generation makes independently one jump from $\cS$, and splits into a random number of new individuals. All individuals exiting the quarter plane are killed. See Figure~\ref{fig:rainbow} for an illustration.

The mean number of visits made by the branching random walk at a fixed point $(i,j)$ starting from the origin is given by
\begin{equation*}
    G^{\textnormal{BRW}}_\nu(i,j)=\mathbb E_{(0,0)}\left(\sum \mathds{1}_{\{\textnormal{BRW} = (i,j)\}}\right),
\end{equation*}
where the sum runs over all particles of all generations. The reason why we introduced this model is that there is a strong connection between the Green function of the branching random walk and the $t$-Green function of one single random walk, as introduced in \eqref{eq:def_Green}:
\begin{equation*}
    G^{\textnormal{BRW}}_\nu(i,j) = G(i,j\vert \mathbb E\nu),
\end{equation*}
with $\mathbb E\nu$ denoting the mean value of the distribution $\nu$. Accordingly,
\begin{equation*}
[x^iy^j]Q(x,y,\tfrac{1}{2}) =  G(i,j\vert \tfrac{\#\cS }{2}) = G^{\textnormal{BRW}}_\nu(i,j),
\end{equation*}
provided $\mathbb E\nu=\tfrac{\#\cS }{2}$.
For example, in the case $\cS=\cA$, we can take $\nu$ to be the uniform distribution over $\{1,2\}$, and get $\mathbb E\nu=\frac{3}{2} = \tfrac{\#\cS }{2}$. See Figure~\ref{fig:rainbow}.

This reformulation has the advantage of giving a clear probabilistic interpretation of the coefficients of $Q(x,y,\tfrac{1}{2})$; on the other hand, unlike the interpretation given in Remark~\ref{rem:explanation}, it does not immediately explain why the point $t=\tfrac{1}{2}$ is critical.

\subsection{Weighted models}
\label{subsec:weights}

While the results presented in this paper concern unweighted walks (i.e., each step from the step set $\mathcal S$ is taken with weight $1$), they can be made more general using the following simple idea. Given any step set $\mathcal S$ and positive constants $\alpha_1$, $\alpha_2$ and $\beta$, the generating function $Q(\alpha_1 x,\alpha_2 y ,\beta t)$ turns out to be the generating function $Q_a(x,y,t)$ of the step set $\mathcal S$ with weights
\begin{equation}
\label{eq:weights}
    a_s = \beta  {\alpha_1}^{s_1} {\alpha_2}^{s_2},
\end{equation}
for $s=(s_1,s_2)\in\mathcal S$. Such families of weights are similar to those used in Section~\ref{subsec:Cramer} to construct the Cram\'er transform of the random walk; the only difference is that in \eqref{eq:weights} we have one more parameter, because we don't force the sum of the weights to be $1$.

The weights \eqref{eq:weights} are constructed using three parameters; consequently, given any weights $(w_s)_{s\in\mathcal S}$ on the models $\mathcal A$ and $\mathcal D$ (which both have cardinality three), we can find values of $\alpha_1$, $\alpha_2$ and $\beta$ such that $a_s=w_s$. For the other models with four or five steps, the weights \eqref{eq:weights} are not sufficient to cover all possible weights, but allow us to reduce the number of parameters by three.

A consequence of the above remarks and of the equality $Q(\alpha_1 x,\alpha_2 y ,\beta t)=Q_a(x,y,t)$ of generating functions is that all transcendence results immediately extend to models weighted according to \eqref{eq:weights}.

\section{Explicit expressions and Bernoulli numbers}
\label{sec:case_1/2}

In the case of the models $\cA$, $\cB$ and $\cD$, we can directly get access to the nature of the generating function  $Q(x,y,t)$ via explicit formulas for~$Q$ evaluated at $t=\frac{1}{2}$.
These formulas are expressed in terms of the \emph{Bernoulli numbers} $(B_n)_{n\geq 0}$, classically defined via their exponential generating function:
    \begin{equation} \label{eq:EGF-Bernoulli}
    \sum_{n\geq 0}B_n\frac{x^n}{n!}=\frac{x}{e^x-1}.
    \end{equation}
We recall that the EGF~\eqref{eq:EGF-Bernoulli} is D-algebraic but that the ordinary generating function
$\sum_{n\geq 0}B_n {x^n}$
is D-transcendental~\cite[Prop.~11]{bostan-divizio-raschel-BELL}, and even strongly D-transcendental~\cite[Thm.~3]{bostan-divizio-raschel-BELL}.
Note that Bernoulli numbers admit nice closed-form expressions (see for instance~\cite{Gould72}), such as
    \begin{equation} \label{eq:closed-Bernoulli}
B_n = \sum_{k=0}^n (-1)^k \binom{n+1}{k+1} \frac{0^n + 1^n + \cdots + k^n}{k+1},
	\end{equation}
therefore Theorems~\ref{thm:A-Bernoulli}, ~\ref{thm:D-Bernoulli} and ~\ref{thm:B-Bernoulli} below
provide closed-form expressions for
$Q_{\cA}(x,0,\tfrac 12)$,
$Q_{\cB}(x,0,\tfrac 12)$
and
$Q_{\cD}(x,0,\tfrac 12)$.

We leave as natural questions for further investigation whether $Q_{\cC}(x,0,\tfrac 12)$ and $Q_{\cE}(x,0,\tfrac 12)$ can equally
be expressed in terms of Bernoulli numbers, and for which models $\cS$ can the generating function $Q_{\cS}(x,0,q/(1+q^2))$ be
expressed in terms of some $q$-deformations of the Bernoulli numbers.

\subsection{Model $\cA$}
\begin{theorem}\label{thm:A-Bernoulli}
In the case of model $\cA$, we have:
  \begin{equation} \label{eq:formula-A-Bernoulli}
    Q_{\cA}(x,0,\tfrac 12)= 2 \cdot \sum_{n \geq 0}   (2^{2n+2}-1) \frac{(-1)^n}{n+1} B_{2n+2} x^{2n}.
  \end{equation}
In particular:
\begin{enumerate}
	\item[(a)] $Q_{\cA}(x,0,\tfrac 12)$ is strongly D-transcendental;
	\item[(b)] the coefficients sequence $(a_n)_{n\geq 0} = (1,0,{\frac{1}{2}},0,1,0,{\frac{17}{4}},0,31,0,{\frac{691}{2}},\ldots)$ of $Q_{\cA}(x,0,\tfrac 12)$ satisfies
\[ 2 \, a_{n+1} = \sum_{\ell=1}^n \binom{n}{\ell} a_{\ell-1} a_{n-\ell} \quad \text{for all} \; n \geq 0 .	\]
\end{enumerate}
\end{theorem}

\begin{proof}
Since equation~\eqref{eq:ModA1/2} admits a unique solution in $\mathbb{C}[[x]]$,
namely $\frac{x^2}{2} \cdot Q_{\cA}(x,0,\tfrac 12)$,
it is enough to prove that the power series $f(x):=\sum_{n \geq 0}   (2^{2n+2}-1) \frac{(-1)^n}{n+1} B_{2n+2} x^{2n+2}$ also satisfies \eqref{eq:ModA1/2}.
\par
We introduce
$T(x)  := \sum_{n \geq 0} T_n x^{2n+1}$, with $T_n = 2^{2n+1} (2^{2n+2}-1) \frac{(-1)^n}{n+1} B_{2n+2}$.
It follows from \cite[\S2.2.1]{bostan-divizio-raschel-BELL}, and in particular from Eq.~(2.7) in {\it loc.\ cit.}, that
$T(x) = x +2 x^{3}+16 x^{5}+272 x^{7}+7936 x^{9}+ \cdots$ satisfies the difference equation
\[
T \left( {\frac {x}{1+2\,ix}} \right) + \left( 1+2\,ix \right) T \left( x \right) =2\,x.
\]
Since
$f(2x) = 2x T(x)$, the ordinary generating function
$f$ satisfies the equation
\[
f \left( {\frac {x}{1+\,ix}} \right) + f \left(  x \right) = \frac{x^2}{1+\,ix},
\]
which coincides with \eqref{eq:ModA1/2}.
Then, assertion (a) follows from the strong D-transcendence of $\sum_{n\geq 0}B_n {x^n}$ and from the closure properties of D-algebraic power series, while assertion (b) follows from the differential equation $\tan'(x) = 1+ \tan^2(x)$ satisfied by the tangent function $\tan(x) = \sum_{n \geq 0} \frac{T_n}{(2n+1)!} {x^{2n+1}}$.
\end{proof}

\begin{remark}\label{rem:24}
\normalfont
Let us prove that, for any $r\in \mathbb{N}$, the power series $\sum_{n \geq 0} \frac{a_n}{(n+r)!} x^n$ is D-algebraic.

Since $G_1(x) = \frac{x^2}{2} \cdot Q_{\cA}(x,0,\tfrac 12)$ satisfies~\eqref{eq:ModA1/2}, it follows that
$A(x) :=  \sum_{n \geq 0} a_n x^n$ satisfies the difference equation
\[ A(x/(1+ix)) + (1+ix)^2 A(x) = 2 (1+ix).\]
A few elementary manipulations show that this is equivalent to the equation
\begin{equation} \label{eq:DeltaPhi}
		\Delta (A(x)) + \Phi^{-i} ( \Delta(A(x))) = -i A(x) + i \Phi^{-i} (A(x)),
\end{equation}
where $\Delta : \mathbb{C}[[x]] \rightarrow \mathbb{C}[[x]]$  is the divided difference operator
\( \Delta (f(x)) := \frac{f(x)-f(0)}{x} \), and where $\Phi^{-i} : \mathbb{C}[[x]] \rightarrow \mathbb{C}[[x]]$ is the Euler-type transform
\( \Phi^{-i} (f(x)) := \frac{1}{1+ix} \cdot f\left( \frac{x}{1+ix} \right) \).
By applying the formal Borel transform $\mathcal{B} : \mathbb{C}[[x]] \rightarrow \mathbb{C}[[x]]$, defined by $ g = \sum_{n \geq 0} g_n x^n \mapsto \widehat{g} := \sum_{n \geq 0} \frac{g_n}{n!} x^n$, to both sides of~\eqref{eq:DeltaPhi}, and by using Lemma 5 in~\cite{bostan-divizio-raschel-BELL}, we deduce that
\begin{equation*} \label{eq:BorelDeltaPhi}
		\widehat{A}(x)' + e^{-ix} \widehat{A}(x)' = -i \widehat{A}(x) + ie^{-ix} \widehat{A}(x),
\end{equation*}
from which it follows that
\begin{equation*} \label{eq:BorelDeltaPhi2}
		\widehat{A}(x)' = \tan(x/2) \widehat{A}(x),
\end{equation*}
and therefore that $\widehat{A}(x) = \sec(x/2)^2$. This proves our assertion for $r=0$.

Let us prove it for $r=1$, that is for $F_A(x) := \sum_{n \geq 0} \frac{a_n}{(n+1)!} x^n$.
Since $\widehat{A}(x) = (x \cdot F_A(x))'$, we deduce that $F_A(x) = \frac{2}{x} \tan \! \left(\frac{x}{2}\right)$, which is indeed D-algebraic and satisfies
\begin{equation*} x^2 F_A(x)^2 - 4x F_A'(x) - 4F_A(x) + 4 = 0.\end{equation*}

For general $r \geq 0$, the proof is similar: denoting $\sum_{n \geq 0} \frac{a_n}{(n+r)!} x^n$ by $F_{A,r} (x)$, we have that $\widehat{A}(x) = (x^r \cdot F_{A,r}(x))^{(r)}$, and the D-algebraicity of
$F_{A,r} (x)$ follows from that of $\widehat{A}(x) = F_{A,0} (x)$.

For instance,
\[ F_{A,2} (x) = \frac{2 \ln \! \left(1+\tan^{2}\left(\frac{x}{2}\right)\right)}{x^{2}} \]
satisfies the algebraic differential equation
\[4-4 x^{2} F_{A,2}''(x)-16 x F_{A,2}'(x)+x^{4} F_{A,2}'(x)^{2}-8 F_{A,2} \! \left(x \right)+4 x^{3} F_{A,2} \! \left(x \right) F_{A,2}'(x)+4 x^{2} F_{A,2} \! \left(x \right)^{2}  = 0.
\]

\end{remark}

\subsection{Model $\cD$}

\begin{theorem} \label{thm:D-Bernoulli}
In the case of model $\cD$, we have:
    \begin{equation} \label{eq:formula-D-Bernoulli}
    Q_{\cD}(x,0,\tfrac 12)= 2 \, \sum_{n \geq 0}   (2n+3) B_{2n+2} (-x)^{n}.
    \end{equation}
In particular:
\begin{enumerate}
	\item[(a)] $Q_{\cD}(x,0,\tfrac 12)$ is strongly D-transcendental;
	\item[(b)] the coefficients sequence $(d_n)_{n\geq 0}
	= (1,\frac13,\frac13,\frac35,\frac53,\frac{691}{105},35,\frac{3617}{15},\ldots)$
	of $Q_{\cD}(x,0,\tfrac 12)$ satisfies
\[ 4(n+2)(2n+3) \, d_{n} = \sum_{\ell=0}^{n-1} \binom{2n+4}{2\ell+3} d_{\ell} d_{n-1-\ell} \quad \text{for all} \;\; n \geq 1 .	\]
\end{enumerate}
\end{theorem}

\begin{proof}
Recall that the power series $Q_{\cD}(x,y,t)$ satisfies the functional equation
	\[ Q_{\cD}(x,y,t) = 1 + t \left( y + \frac{y}{x} + \frac{x}{y} \right) Q_{\cD}(x,y,t) - t \frac{x}{y} Q_{\cD}(x,0,t) - t \frac{y}{x} Q_{\cD}(0,y,t),	\]
therefore the following ``kernel equation'' holds
	\begin{equation} \label{eq:redkerD}
		  \frac{x^2}{2} Q_{\cD}(x,0,\tfrac12) + \frac{{y}^2}{2} Q_{\cD}(0,y,\tfrac12) = xy ,
	\end{equation}
whenever $(x,y)$ is a point on the ``kernel curve'' $x \,y^{2} +  x^{2} + y^{2} = 2x y$.

This kernel curve has genus zero, and it can be rationally parametrized by
\[ x(s) =-s^2, \qquad y(s) = - \frac{s^2}{s+1}  \]
and by
\[ \widetilde{x}(s) = -\frac{s^2}{(s+1)^2} = x \left( \frac{s}{s+1}\right), \qquad y(s) =  - \frac{s^2}{s+1} . \]
(Notice that these parametrizations are different from the one used in the proof of Theorem~\ref{thm:main-Galois}.)

From~\eqref{eq:redkerD} with $(x,y)$ replaced by $(x(s),y(s))$, respectively by $(\widetilde{x}(s),y(s))$, we deduce by subtraction that the power series $W(x):=\frac{x^2}{2} Q_{\cD}(-x^2,0,\tfrac12)$ is the unique solution in $\mathbb{Q}[[x]]$ of the functional equation
\begin{equation} \label{eq:WD}
W(x) - \frac{1}{(x+1)^2} \cdot W\left( \frac{x}{x+1}\right) = \frac{x^3(x+2)}{(x+1)^3} .
\end{equation}
Let us introduce the power series $\beta(x) := \sum_{n \geq 0} B_n x^{n+1} = x - \frac{x^2}{2} + \frac{x^3}{6} - \frac{x^5}{30} +\cdots.$ Since $B_{2k+1}= 0$ for $k>0$, we have that $\beta'(x) = 1-x + \sum_{n \geq 1} (2n+1) B_{2n} x^{2n}$.
To prove identity~\eqref{eq:formula-D-Bernoulli} it is enough to show that the power series $V(x) := \beta'(x) + x - 1$ also satisfies equation~\eqref{eq:WD}. Indeed, the fact that $V(x) = W(x)$ implies that
\[
Q_{\cD}(-x^2,0,\tfrac12) = \left( \beta'(x) + x - 1 \right) \cdot \frac{2}{x^2} = 2 \cdot \sum_{n \geq 1} (2n+1) B_{2n} x^{2n-2},
\]
which is equivalent to~\eqref{eq:formula-D-Bernoulli}.  Finally, the fact that $V(x)$ satisfies~\eqref{eq:WD} is a direct consequence of the functional equation $\beta'(x) - \frac{1}{(x+1)^2} \cdot \beta'(x/(x+1)) = 2x/(x+1)^3$, which is implied by differentiating the equality $\beta(x) - \beta(x/(x+1)) = x^2/(x+1)^2$, itself equivalent to the first entry in Table 2 in~\cite{bostan-divizio-raschel-BELL} (see also \cite[Eq.~(A.7), p.~241]{AIK1024}).

Finally, assertion (a) follows from the strong D-transcendence of $\beta(x)$ and assertion (b) follows from Euler's quadratic recurrence $(2n+1) B_{2n} + \sum_{i=1}^{n-1} \binom{2n}{2i} B_{2i} B_{2n-2i} = 0$ for all $n>1$.
\end{proof}

\begin{remark}
\normalfont
A similar approach as in Remark~\ref{rem:24} can be used to show that for any $r\in \mathbb{N}$, the power series $F_{D,r}(x) := \sum_{n \geq 0} \frac{d_n}{(2n+r)!} x^n$ is D-algebraic.
For instance,
\[ F_{D,3}(x) = \frac{2-\sqrt{x}\, \cot \! \left(\frac{\sqrt{x}}{2}\right)}{x} \]
satisfies the algebraic differential equation
\[ x F_{D,3}(x)^2 - 4x F_{D,3}'(x) - 6F_{D,3}(x) + 1 = 0,\]
while
\[ F_{D,2}(x) = \frac{-4 \sqrt{x}\, \cot \! \left(\frac{\sqrt{x}}{2}\right)+4+x \left(\csc^{2}\left(\frac{\sqrt{x}}{2}\right)\right)}{2 x}
 \]
satisfies the algebraic differential equation
{\[
8 x F_{D,2}'(x)-4 x^{2} F_{D,2}'(x)^{2}+16 F_{D,2} \! \left(x \right)-16 x  F_{D,2} \! \left(x \right) F_{D,2}'(x)  -\left(x +16\right) F_{D,2} \! \left(x \right)^{2}+2 x F_{D,2} \! \left(x \right)^{3} = 4.
\]
}
\end{remark}

\subsection{Model $\cB$}

Let $(M_{n})_{n\geq 0} = (1, 2, 8, 56, 608,\ldots)$
be the sequence of \textit{median Genocchi numbers} (\href{https://oeis.org/A005439}{A005439});
in the notation of \cite{DuZe94}, $M_n= H_{2n+3}$.

\begin{theorem}
	\label{thm:B-Bernoulli}
For model $\mathcal{B}$ we have
\begin{equation} \label{eq:formula-B-Bernoulli}
Q_{\cB}(x,0,\tfrac{1}{2}) =
\sum_{n \geq 0}
\frac{M_{n}}{2^n}
\, x^{n}
=
2 \cdot \sum_{n=0}^{\infty}
\left(-\frac{x}{2}\right)^{n}
\sum _{k=0}^{n+1}\binom{n+1}{k}
\left( {2}^{n+k+2} -1 \right) B_{n+k+2}.
\end{equation}
Moreover, $Q_{\cB}(x,0,\tfrac{1}{2})$ is strongly differentially transcendental.
\end{theorem}

\begin{proof}
The kernel curve defined by the polynomial $K(x,y,\frac12) = xy - \frac12 \left( x^2y+xy^2+x^2+y^2\right)$
can be parametrized by
\[
x(s) = -\frac{2s^2}{s+1} \quad\text{and}\quad y(s) = -\frac{2s^2}{1-s}.
\]
(Notice that, again, the parametrization used below is not the one used in
the proof of Theorem~\ref{thm:main-Galois}.)

By plugging this parametrization into the kernel equation
\[K\left(x,y,\tfrac12\right) Q\left(x,y,\tfrac12\right) = xy - \frac{x^2}{2} Q\left(x,0,\tfrac12\right) - \frac{y^2}{2} Q\left(0,y,\tfrac12\right)\]
and by using the equality $Q(x,0,t) = Q(0,x,t)$,
we obtain that the power series $f(s) := Q(s,0;\frac12)$ satisfies  the equation
\begin{equation} \label{eq0_B}
	 \frac{1}{(s+1)^2} \, f\left( -\frac{2s^2}{s+1}\right) +
	 \frac{1}{(1-s)^2} \, f\left( -\frac{2s^2}{1-s}\right) =
       \frac{2}{1-s^2}   .
\end{equation}
Since the above equation~\eqref{eq0_B} admits a unique solution in $\mathbb{C}[[s]]$,
it is enough to prove that
 \[
 f(s) :=
\sum_{n \geq 0}
\frac{M_{n}}{2^n} \, s^{n}
 = 1 + s + 2s^2 + 7s^3 + 38 s^4 + \cdots
 \]
satisfies~\eqref{eq0_B} as well.
We deduce from \cite[Cor.~2]{DuZe94} that the power series
\[
h(s)
= s - s^2 \sum_{n \geq 0} M_{n} (-s)^{n}
 = s-s^2+2s^3-8s^4 + \cdots
\]
satisfies the equation
\begin{equation} \label{eq-h}
h \left( {\frac {s^2}{1+\,s}} \right) + h \left( {\frac {s^2}{1-\,s}} \right)  = 2s^2.
\end{equation}
From there, using $h(s) = s - s^2 f(-2s)$, it follows that
$f(s)$
does indeed satisfy~\eqref{eq0_B}.

The expression in terms of Bernoulli numbers is a consequence of
the identity
\[
M_{n-1} =
2 (-1)^{n+1}   \sum _{k=0}^{n}\binom{n}{k}
\left( {2}^{n+k+1} -1 \right) B_{n+k+1} \quad \text{for all } n \geq1,
\]
itself a consequence of \cite[Cor.~3]{DuZe94},
\[
M_{n-1} = \frac{1}{4^n} \cdot \sum _{k=0}^{n}
\binom{n}{k}
 (2k+1) E_{2k} \quad \text{for all } n \geq1 ,
\]
where $(E_{2n})_{n\geq0} = (1, 1, 5, 61, 1385, \ldots)$
(\href{https://oeis.org/A000364}{A000364})
is the sequence of Euler (``secant'') numbers.

Finally, the strong differential transcendence of $Q_{\cB}(x,0,\tfrac{1}{2})$ can be proved as follows.
By making the change of unknown $g(s) := h(s^2/(1+s)) = s^{2}-s^{3}+s^{5}-3 s^{7}+17 s^{9}-155 s^{11}+ \cdots$, equation~\eqref{eq-h} becomes
\(
g(s) + g\bigl(\frac{s}{1-s}\bigr) = 2s^2
\)
and Theorem 3 in \cite{bostan-divizio-raschel-BELL} concludes.

Note that the coefficients of $g$ are the
\emph{Genocchi numbers} %
(\href{https://oeis.org/A001469}{A001469}) and the strong D-transcendence of their OGF was explicitly mentioned in the list on page 6 of~\cite{bostan-divizio-raschel-BELL}.
\end{proof}

\begin{remark}
\normalfont
Contrary to (the proofs of) Theorems~\ref{thm:A-Bernoulli} and~\ref{thm:D-Bernoulli}, we cannot infer from the proof of Theorem~\ref{thm:B-Bernoulli} that some exponential version of $Q_{\cB}(x,0,\tfrac{1}{2})$ is D-algebraic. Similarly, we are unable to deduce from Eq.~\eqref{eq:formula-B-Bernoulli} a nonlinear quadratic recurrence similar to those in part (b) of  Theorems~\ref{thm:A-Bernoulli} and~\ref{thm:D-Bernoulli}.
However, from~\eqref{eq0_B} we can deduce the following linear recurrence (of unbounded length)
on the coefficients $b_n:=[x^n]\,Q_{\cB}(x,0,\frac12)$
\[ \sum_{ \ell = 0} ^n (-2)^\ell  \binom{2n-\ell+1}{\ell+1} b_\ell = 1, \quad \text{for all } \; n \geq 0.\]
By Theorem~\ref{thm:B-Bernoulli}, this identity is equivalent to the second identity in~\cite[Cor.~1]{DuZe94}.
\end{remark}

\begin{remark}
\normalfont
Contrary to models $\cA$ and $\cD$, the power series $Q_{\cB}(x,0,\tfrac{1}{2})$ has \emph{integer} coefficients.
This is a consequence of Theorem~\ref{thm:B-Bernoulli}
and of a result due to Barsky and Dumont~\cite[Thm.~4]{BarskyDumont81}.
The coefficient sequence $(b_n)_{n\geq0} = (1, 1, 2, 7, 38, 295, 3098, \ldots)$ of $Q_{\cB}(x,0,\tfrac{1}{2})$ appears to coincide with the so-called \emph{Dellac sequence} (\href{https://oeis.org/A000366}{A000366}),
whose $n$-th term counts the ways to partition $\{1,2,\dots,2n\}$
into $n$ pairs $P_1,\dots,P_n$ with $P_j\subseteq\{j,j+1,\dots,j+n\}$ for all~$j$~\cite[\S IV]{Dellac01}.
 Our interpretation in terms of walks for model $Q_{\cB}$ seems to be new.
It would be interesting to have a direct understanding of this new interpretation.
\end{remark}

\begin{remark}
\normalfont
It can be shown that $Q_{\cB}(x,0,-\tfrac{1}{2}) = -2 \cdot \sum_{n \geq 0} B_{n+1} x^n$.
Indeed, the kernel at $t=-1/2$ factors $(x+y)(xy+x+y)/2$, and the unique power series solution of $x^2 f(x) + (-x/(1+x))^2 f(-x/(1+x)) = 2x^2/(x+1)$ is $f(x) = -2 \cdot \sum_{n \geq 0} B_{n+1} x^n$.
\end{remark}

\section{Conclusion and open questions}

One of the principal motivations of the present work was to revisit the paper \cite{DHRS0}, which establishes, for the five singular models of Table~\ref{tab:list}, the differential transcendence of the generating function $Q(x,y,t)$ by proving that it is differentially transcendental with respect to $x$ and $y$ for generic values of $t$. In contrast, our approach shows that it is sufficient to consider a single specialization, namely $t=\frac{1}{2}$, in order to derive this property. More precisely, we proved that the function $Q(x,y,\frac{1}{2})$ is in fact \textit{strongly} differentially transcendental (Theorem~\ref{thm:main-Galois}). Once the unexpected connection with Bernoulli numbers had been established (see e.g.\ Theorem~\ref{thm:A-Bernoulli}), this phenomenon became rather natural, since the arithmetic nature of the ordinary generating function of Bernoulli numbers is known by \cite{bostan-divizio-raschel-BELL}. We further completed the picture by proving that $Q(x,y,t)$ is differentially transcendental in $x$ and $y$ for every $t\in(0,\frac12]$, and not merely for generic ones (see again Theorem~\ref{thm:main-Galois}).

\medskip

These results reveal two main features. First, the nature of differential transcendence undergoes a transition at the critical value $t=\frac{1}{2}$, changing from a classical form to a strong form. Second, we showed that this transition admits a probabilistic interpretation (Section \ref{sec:probab}): a certain exponential transform associated with the jumps becomes singular at this critical point. Remarkably, it is precisely at $t=\frac{1}{2}$ that special sequences such as the Bernoulli numbers naturally emerge.

\medskip

In this concluding section, we would like to outline several perspectives opened by these results. The first direction would be to seek a more direct combinatorial explanation for the appearance of Bernoulli numbers at $t=\frac{1}{2}$.

\medskip

We now comment on the robustness of our main results. Although the singularity assumption is essential (indeed, no analogous phenomenon seems to occur for non-singular step sets), similar statements should remain valid for a broad class of singular random walks. In dimension two, we established such results for all unweighted singular models, and certain weighted variants can also be incorporated (Section~\ref{subsec:weights}). One may furthermore consider larger jumps, for which analogous conclusions are expected to hold. More generally, we believe that related phenomena should also arise in higher dimensions. The core of our arguments relies on the recursive structure of the models, and this structure is preserved under several natural generalizations. Extending our results in this direction is however far from immediate and constitutes an interesting topic for future research.

\medskip

Another natural direction concerns deformations of Bernoulli numbers, such as the $q$-Bernoulli numbers. It would be interesting to investigate whether singular models contain not only the classical Bernoulli numbers in their generating functions at $t=\frac{1}{2}$, but also some of their deformations. Indeed, there are several ways to enrich the present framework by replacing a sequence of numbers with a triangular array. For instance, instead of restricting attention to walks ending at points of the form $(k,0)$ on the horizontal axis, one could study walks terminating at arbitrary points of the segment $S_k$ in \eqref{eq:segments-intro}. Another possibility would be to refine the models by recording the number of boundary contacts such as in \cite{Beaton19,Bonnet26}. In such settings, the Bernoulli numbers would naturally be replaced by deformed families reducing to the classical ones in a suitable limit.

\medskip

Finally, the present paper provides new examples of combinatorial models whose generating functions are strongly differentially transcendental.
We hope to pursue this line of investigation further, both by identifying additional families of models exhibiting this property and, more broadly, by deepening the connections between combinatorics and transcendence theory.

\appendix
\section*{Appendix}
\section{Proof of Theorem~\ref{thm:main-Galois}
for models $\cB$, $\cC$, $\cD$ and $\cE$}

We start detailing the proof of Theorem~\ref{thm:main-Galois} for models $\mathcal{B}$ and $\mathcal{D}$, which are essentially similar.
Finally we look at models $\mathcal{C}$ and $\mathcal{E}$, which are a bit more complicated.

\subsection{Proof of Theorem~\ref{thm:main-Galois} in the case of models \texorpdfstring{$\cB$}{B} and  \texorpdfstring{$\cD$}{D}}

\begin{proof}[Proof (Model \texorpdfstring{$\cB$}{B})]
We have:
    \begin{align}
          K(x,y) &= xy - t\left({y}^{2}+x{y}^{2}+{x}^{2}y+{x}^{2} \right)\,;\nonumber\\
          \displaystyle x_0(s)&=-\frac{\left(v +1\right) \left(1-v \right)^{2} s}{v \left(s -1\right) \left(v s -1\right)}\quad \text{and}\quad
            y_0(s)=-\frac{\left(v +1\right) \left(1-v \right)^{2} s}{\left(s \,v^{2}-1\right) \left(v s -1\right)}\,;
            \label{eq:ParaBdhrs}\\
      \displaystyle \widetilde x_0(s)&=-\frac{v s \left(v +1\right) \left(1-v \right)^{2}}{\left(s \,v^{3}-1\right) \left(s \,v^{2}-1\right)}\,, \quad
            \hbox{with $\widetilde x_0(s)=x_0(v^2\,s)$.}\nonumber
    \end{align}
Therefore, we deduce from \eqref{eq:functional-equation} that the formal
power series $G_0(s)$
in \eqref{eq:def_G0} satisfies the $v^2$-difference equation:
    \begin{multline*}
     G_0(s)-G_0(v^2\,s)=\frac{-v^{4}-v^{3}+2 v^{2}+v -1}{\left(v s -1\right) v^{2}}+\frac{v^{2}-1}{\left(s -1\right) v}\\
    +\frac{v^{2}-1}{\left(s \,v^{3}-1\right) v}+\frac{-v^{4}+2 v^{2}-1}{\left(v s -1\right)^{2} v^{2}}+\frac{v^{4}-2 v^{2}+1}{\left(s \,v^{2}-1\right)^{2} v^{2}}+\frac{v^{4}-v^{3}-2 v^{2}+v +1}{\left(s \,v^{2}-1\right) v^{2}}\,.
    \end{multline*}
If we set
    \[
    \widetilde{G}_0:=G_0
    +\frac{v^{4}-2 v^{2}+1}{\left(s -1\right)^{2} v^{2}}+\frac{v^{4}-v^{3}-2 v^{2}+v +1}{\left(s -1\right) v^{2}}+\frac{v^{2}-1}{\left(v s -1\right) v}\,,
    \]
then the functional equation becomes
    \begin{multline}\label{eq:feqBv<1}
    \widetilde G_0(s)-\widetilde G_0(v^2\,s)=\frac{-v^{4}-v^{3}+2 v^{2}+v -1}{\left(v s -1\right) v^{2}}+\frac{v^{2}-1}{\left(s -1\right) v}\\
    +\frac{-v^{4}+2 v^{2}-1}{\left(v s -1\right)^{2} v^{2}}+\frac{v^{2}-1}{\left(v s -1\right) v}+\frac{v^{4}-2 v^{2}+1}{\left(s -1\right)^{2} v^{2}}+\frac{v^{4}-v^{3}-2 v^{2}+v +1}{\left(s -1\right) v^{2}}.
    \end{multline}
For any $v\in(0,1)$, the right-hand side of the equation has only one pole per $v^2$-orbit, hence the functional equation has no rational solution.
\par
For $v=1$, a parametrization of $K(x,y,\frac12)$  is given by:
    \begin{equation}\label{eq:paraBv=1}
    x_1(s)=-\frac{s^{2}}{s +2}\quad\text{and}\quad
    y_1(s)=-\frac{s^{2}}{\left(s +1\right) \left(s +2\right)}.
    \end{equation}
Moreover $\widetilde{x}_1(s):=   -\frac{s^{2}}{(s+1)(3s+2)}$ and $\widetilde x_1(s)=x_1\l({\frac {s}{s+1}}\r)$.
Therefore, the formal power series
$G_1(s)$ in \eqref{eq:G1}
satisfies the
difference equation:
    \[
        G_1(s)-G_1\l({\frac {s}{s+1}}\r)
        =s-\frac{16}{3}-\frac{16}{\left(s +2\right)^{2}}-\frac{4}{3 \left(3 s +2\right)}+\frac{20}{s +2}+\frac{1}{\left(s +1\right)^{2}}-\frac{1}{s +1}\,.
    \]
If we set $\widetilde G_1(s)=G_1(s)+\frac{16}{\left(s +2\right)^{2}}-\frac{20}{s +2}$,
then
    \begin{equation}\label{eq:feqBv=1}
    \widetilde G_1(s)-\widetilde G_1\l(\frac {s}{s+1}\r)=
    s -\frac{4}{9}+\frac{1}{\left(s +1\right)^{2}}+\frac{16}{9 \left(3 s +2\right)}-\frac{16}{9 \left(3 s +2\right)^{2}}-\frac{1}{s +1}.
    \end{equation}
One verifies that the iterates of $1/(s+1)$ and of $1/(3s+2)$ with respect to the homography $s\mapsto \frac{s}{s+1}$
are respectively:
    \[
    \left\{\frac{ns+1}{(n+1)s+1}\,,~n\geq 0\right\}\quad\hbox{and}\quad
    \left\{\frac{ns+1}{(2n+3)s+2}\,,~n\geq 0\right\},
    \]
which implies that the poles of the right-hand side of the equation above
are not on the same  orbit, and hence
that $G$ cannot be a rational function.
\end{proof}

\begin{proof}[Proof (Model \texorpdfstring{$\cD$}{D})]
We have:
    \begin{align}
        K(x,y)&=xy-t\l(y^2+xy^2+x^2\r)\,;\nonumber\\
      \displaystyle      x_0(s)&=\frac{s \left(v -1\right)^{2} \left(v +1\right)^{2}}{v^{2} \left(s +1\right)^{2}}\quad \text{and}\quad
            y_0(s)=\frac{s \left(v -1\right)^{2} \left(v +1\right)^{2}}{v \left(s +1\right) \left(s \,v^{2}+1\right)}\,;
            \label{eq:ParaDdhrs}\\
      \displaystyle   \widetilde x_0(s)&=\frac{s \left(v -1\right)^{2} \left(v +1\right)^{2}}{\left(s \,v^{2}+1\right)^{2}}\,,\quad
            \hbox{with $\widetilde x_0(s)=x_0(v^2\,s)$.}\nonumber
    \end{align}
Therefore, we deduce from \eqref{eq:functional-equation} that the formal
power series $G_0(s)$ in \eqref{eq:def_G0} satisfies the $v^2$-difference equation:
    \begin{multline*}
     G_0(s)-G_0(v^2\,s)=
     \frac{-v^{6}+3 v^{4}-3 v^{2}+1}{\left(s +1\right)^{3} v^{3}}
     +\frac{v^{4}-2 v^{2}+1}{\left(s +1\right) v^{3}}
     +\frac{v^{6}-4 v^{4}+5 v^{2}-2}{\left(s +1\right)^{2} v^{3}}\\
     +\frac{-v^{4}+2 v^{2}-1}{\left(s \,v^{2}+1\right) v}
     +\frac{-v^{6}+3 v^{4}-3 v^{2}+1}{\left(s \,v^{2}+1\right)^{3} v^{3}}
     +\frac{2 v^{6}-5 v^{4}+4 v^{2}-1}{\left(s \,v^{2}+1\right)^{2} v^{3}}.
    \end{multline*}
If we set
    \[
    \widetilde{G}_0:=G_0
    -\frac{-v^{6}+3 v^{4}-3 v^{2}+1}{\left(s +1\right)^{3} v^{3}}
    -\frac{v^{4}-2 v^{2}+1}{\left(s +1\right) v^{3}}
    -\frac{v^{6}-4 v^{4}+5 v^{2}-2}{\left(s +1\right)^{2} v^{3}}\,,
    \]
then the functional equation becomes
    \begin{multline}\label{eq:feqDv<1}
    \widetilde G_0(s)-\widetilde G_0(v^2\,s)=\frac{-v^{6}+3 v^{4}-3 v^{2}+1}{\left(s \,v^{2}+1\right) v^{3}}+\frac{-2 v^{6}+6 v^{4}-6 v^{2}+2}{\left(s \,v^{2}+1\right)^{3} v^{3}}+\frac{3 v^{6}-9 v^{4}+9 v^{2}-3}{\left(s \,v^{2}+1\right)^{2} v^{3}}
.    \end{multline}
For any $v\in(0,1)$, the right-hand side of the equation has only one pole per $v^2$-orbit, hence the functional equation has no rational solution.
\par
For $v=1$, $K(x,y,\frac12)=0$ can be parametrized by:
\begin{equation}\label{eq:paraDv=1}
    x_1(s)=-\frac{4 s^{2}}{\left(s -1\right)^{2}}\quad \text{and}\quad
    y_1(s)=\frac{4 s^{2}}{\left(s +1\right) \left(s -1\right)}.
\end{equation}
Moreover $\widetilde{x}_1(s)= -\frac{4 s^{2}}{\left(s +1\right)^{2}}=x_1\l(\frac{s}{2s+1}\r)$.
Therefore, the formal power series
$G_1(s)$ in \eqref{eq:G1} satisfies the
difference equation:
    \[
        G_1(s)-G_1\l({\frac {s}{2s+1}}\r)
        =-\frac{8}{\left(1 +s\right)^{3}}+\frac{8}{\left(1 -s\right)^{3}}+\frac{28}{\left(1+s\right)^{2}}-\frac{28}{\left(1-s\right)^{2}}-\frac{32}{1+s}+\frac{32}{1-s}.
    \]
If we set $\widetilde G_1(s)=G_1(s)+\frac{8}{\left(s -1\right)^{3}}+\frac{28}{\left(s -1\right)^{2}}+\frac{32}{s -1}$,
then
\begin{equation}\label{eq:feqDv=1}
    \widetilde G_1(s)-\widetilde G_1\l(\frac {s}{2s+1}\r)=
    16-\frac{16}{\left(s +1\right)^{3}}+\frac{48}{\left(s +1\right)^{2}}-\frac{48}{s +1}.
\end{equation}
Since the right-hand side has only one pole,
$\widetilde G_1$ and $G_1$ cannot be rational functions.
\end{proof}

\subsection{Proof of Theorem~\ref{thm:main-Galois} in the case of models \texorpdfstring{$\cC$}{C} and
\texorpdfstring{$\cE$}{E}}

\begin{proof}[Proof (Model \texorpdfstring{$\cC$}{C})]
The data of this model are:
    \begin{align}
      K(x,y)&=xy-t\l(y^2+xy^2+x^2y^2+x^2y+x^2\r)\,;\nonumber\\
        \displaystyle      \label{eq:ParaCmaple}
            x_0(s)&=\frac{\left(v +1\right) \left(v -1\right)^{2} s}{\left(s^{2} +(v+1)s + v^2 -v+1\right) v}\quad \text{and}\quad
            y_0(s)=\frac{\left(v +1\right) \left(v -1\right)^{2} s}{v^2s^2+(v^2+v)s+v^2-v+1}\,;\\
        \displaystyle   \widetilde x_0(s)&=\frac{v s \left(v +1\right) \left(v -1\right)^{2}}{v^{4}s^{2} +(v^{3}+v^2)s+ v^{2}-v+1}\,,
            \quad\hbox{with $\widetilde x_0(s)=x_0(v^2\,s)$.}\nonumber
    \end{align}
As in the previous cases, we find the functional equation:
    \begin{multline*}
     G_0(s)-G_0(v^2\,s)
     =\frac{s \,v^{5}-s \,v^{3}+v^{4}-v^{3}+v -1}{\left(s^{2} v^{4}+s \,v^{3}+v^{2} s +v^{2}-v +1\right) v}\\
     +\frac{-s \,v^{4}-s \,v^{3}-2 v^{4}+v^{2} s +2 v^{3}+v s -2 v +2}{\left(s^{2} v^{2}+v^{2} s +v s +v^{2}-v +1\right) v}+\frac{v^{4}+v^{2} s -v^{3}-s +v -1}{\left(s^{2}+v s +v^{2}+s -v +1\right) v}\,.
    \end{multline*}
If we set $\widetilde{G}_0(s)=G_0(s)-\frac{v^{4}+v^{2} s -v^{3}-s +v -1}{\left(s^{2}+v s +v^{2}+s -v +1\right) v}$, then
we obtain the functional equation:
\begin{equation}\label{eq:feqCv=1}
     \widetilde{G}_0(s)-\widetilde{G}_0(v^2\,s)
     =r(s)-r(vs)\,,\quad\hbox{where }
     r(s)=\frac{-s \,v^{4}-s \,v^{3}-2 v^{4}+v^{2} s +2 v^{3}+v s -2 v +2}{\left(s^{2} v^{2}+v^{2} s +v s +v^{2}-v +1\right)v}.
\end{equation}
To conclude it is enough to check that the denominator $p(s):=s^{2} v^{2}+v^{2} s +v s +v^{2}-v +1$
of $r(s)$ has the property that $p(s)$ and $p(v^{2n+1}\,s)$
have no common roots, for any $v\in(0,1)$ and $n\in\Z$. Let us introduce a parameter $a$ and calculate the resultant of
$p(s)$ and $p(v\cdot a\cdot s)$ with respect to $s$:
    \[
    v^{8} \left(v^{2}-v +1\right)\left(a v -1\right)^{2}  \left(a^{2} v^{4}-a^{2} v^{3}+a^{2} v^{2}+a \,v^{3}-4 a \,v^{2}+a v +v^{2}-v +1\right).
    \]
The polynomial above has an obvious double root $a=v^{-1}$ and vanishes
at the following values of $a$:
$\frac{i\sqrt{3}\, v^{2}-i\sqrt{3}-v^{2}+4 v -1}
{2 \left(v^{2}-v +1\right) v}$ and
$-\frac{i\sqrt{3}\, v^{2}-i\sqrt{3}+v^{2}-4 v +1}
{2 \left(v^{2}-v +1\right) v}$. We need to prove that there is no $v\in(0,1)$ and no $n\in\Z$ such that one of the latter is equal to $v^{2n}$.
For the first root, suppose by contradiction that
    \[
    \frac{i\sqrt{3}\, v^{2}-i\sqrt{3}-v^{2}+4 v -1}{2 \left(v^{2}-v +1\right) v}=v^{2n}.
    \]
This is equivalent to say that for $v\in(0,1)$, we have $i\sqrt{3}(v^2-1)=2\,v^{2n+1}\left(v^{2}-v +1\right)+v^{2}-4 v +1$,
which cannot happen because we have a non-zero purely imaginary number of the left-hand side and a real number on the right-hand side.
\par
For $v=1$, the kernel curve defined by $xy-\frac{1}{2}\l(y^2+xy^2+x^2y^2+x^2y+x^2\r)=0$ is parametrized by:
    \begin{equation}\label{eq:paraCv=1}
    x_1(s)=\frac{- s^{2}}{ \frac{s^{2}}{2} - s + 2} \quad \text{and}\quad
    y_1(s)=\frac{- s^{2}}{ \frac{s^{2}}{2} + s + 2}\,.
    \end{equation}
We have $\displaystyle{\widetilde{x}_1(s) = \frac{- s^{2}}{\frac32 s^{2} + 3 s + 2}}$, so that
$\widetilde{x}_1(s)=x_1\l(\frac{s}{1+s}\r)$.
Therefore, the formal power series $G_1(s)$
in \eqref{eq:G1} satisfies the
difference equation:
    \[
     G_1(s)-G_1\l({\frac {s}{s+1}}\r)= \frac{8}{3}
    +\frac{4 (3 s +4)}{3(3 s^2 + 6 s + 4)}
    -\frac{8}{s^2 - 2 s + 4}
    +\frac{4 (s-2)}{s^2 + 2 s + 4}.
    \]
We need to prove that this functional equation has no rational function solutions.
If we set $\widetilde{s}=1/s$ then we obtain:
    \[
     G_1(\widetilde s)-G_1\l(\widetilde s+1\r)=
    -\frac{4 (\widetilde s + 1)}{4 \widetilde s^2 + 6 \widetilde s + 3}
    -\frac{2 (2 \widetilde s - 1)}{4 \widetilde s^2 - 2 \widetilde s + 1}
    +\frac{2 (4 \widetilde s + 1)}{4 \widetilde s^2 + 2 \widetilde s + 1}.
    \]
We define $\widetilde{G}_1(\widetilde s)
=
G_1(\widetilde s)
+ (2 (2 \widetilde s - 1))/(4 \widetilde s^2 - 2 \widetilde s + 1)$, so that:
\begin{equation}\label{eq:feqCv<1}
     \widetilde{G}_1(\widetilde s)-\widetilde{G}_1\l(\widetilde s+1\r)=
    \frac{2 (4 \widetilde s + 1)}{4 \widetilde s^2 + 2 \widetilde s + 1}
    -\frac{2 (4 \widetilde s + 3)}{4 \widetilde s^2 + 6 \widetilde s + 3} .
\end{equation}
The poles of the right-hand side are
    \[
    - \frac{3}{4}+\frac{i \sqrt{3}}{4}\,,~
    - \frac{3}{4}-\frac{i \sqrt{3}}{4}\,,~
    - \frac{1}{4}+\frac{i \sqrt{3}}{4}\,,~
    - \frac{1}{4}-\frac{i \sqrt{3}}{4}\,,~
    \]
Since any pair of them does not differ by an integer, the functional equation above cannot have a rational solution.
We conclude that $G_1(s)$ is differentially transcendental over $\C(\{s\})$.
\end{proof}

\begin{proof}[Proof (Model \texorpdfstring{$\cE$}{E})]
We have
    \begin{align}
          K(x,y)&=xy-t(y^2+x^2y^2+xy^2+x^2)\,;\nonumber\\    \label{eq:ParaEmaple}
            x_0(s)&=\frac{s \left(v -1\right)^{2} \left(v +1\right)^{2}}{v \left(s^{2}+2 v s+v^{4} -v^{2}+1\right)}\quad\text{and}\quad
             y_0(s)=\frac{s \left(v -1\right)^{2} \left(v +1\right)^{2}}{v^2s^2+(v^3+v)s+v^{4} -v^{2}+1}\,;\\
        \displaystyle   \widetilde x_0(s)&=\frac{v s \left(v -1\right)^{2} \left(v +1\right)^{2}}{v^4s^2  +2v^{3} s +v^4-v^{2}+1}\,,\quad
            \hbox{with } \widetilde x_0(s)=x_0(v^2\,s)\,.\nonumber
    \end{align}
As in the previous cases, we find the functional equation:
    \begin{multline*}
     G_0(s)-G_0(v^2\,s)
     =\frac{v^{6}+s \,v^{3}-2 v^{4}-v s +2 v^{2}-1}{\left(v^{4}+s^{2}+2 v s -v^{2}+1\right) v}\\
     +\frac{-s \,v^{5}-2 v^{6}+4 v^{4}+v s -4 v^{2}+2}{\left(s^{2} v^{2}+s \,v^{3}+v^{4}+v s -v^{2}+1\right) v}+\frac{s \,v^{5}+v^{6}-s \,v^{3}-2 v^{4}+2 v^{2}-1}{\left(s^{2} v^{4}+2 s \,v^{3}+v^{4}-v^{2}+1\right) v}\,.
    \end{multline*}
If we set $\widetilde{G}_0(s)=G_0(s)-\frac{v^{6}+s \,v^{3}-2 v^{4}-v s +2 v^{2}-1}{\left(v^{4}+s^{2}+2 v s -v^{2}+1\right) v}$,
we obtain the functional equation:
    \begin{multline}\label{eq:feqEv<1}
     \widetilde{G}_0(s)-\widetilde{G}_0(v^2\,s)
     =\frac{2 s \,v^{5}+2 v^{6}-2 s \,v^{3}-4 v^{4}+4 v^{2}-2}{\left(s^{2} v^{4}+2 s \,v^{3}+v^{4}-v^{2}+1\right) v}+\frac{-s \,v^{5}-2 v^{6}+4 v^{4}+v s -4 v^{2}+2}{\left(s^{2} v^{2}+s \,v^{3}+v^{4}+v s -v^{2}+1\right) v}.
    \end{multline}
To conclude, we need to prove that the roots of the denominator of the right-hand side have distinct $v^2$-orbits, for any $v\in(0,1)$.
The roots are:
\[
    \alpha_1=\tfrac{(v^2-1)i-v}{v^2}\,,~
    \alpha_2=\tfrac{(1-v^2)i-v}{v^2}\,,~\alpha_3=\tfrac{\sqrt{3}(v^2-1)i-(v^2+1)}{2v}\,,~
    \alpha_4= \tfrac{\sqrt{3}(1-v^2)i-(v^2+1)}{2v}\,.
\]
Let us suppose by contradiction that there exist $n\in\Z$ and $v\in(0,1)$, such that $\alpha_1=v^{2n}\,\alpha_3$, i.e.,
$2(v^2-1)i-2v^2=v^{2n+1}\bigl(\sqrt{3}(v^2-1)i-(v^2+1)\bigr)$. Since $v$ is real we must have
$(\sqrt{3}v^{2n+1}-2)(v^2-1)=0$, which is impossible for any $v\in(0,1)$.
Therefore the $v^2$-orbits of $\alpha_1$ and $\alpha_3$ are disjoint for any $v\in(0,1)$.
One can reason in the same way for all the other 5 possible quotients $\alpha_i/\alpha_j$, with $4\geq j>i\geq 1$, to prove that
$\widetilde{G}_0$ (and hence $G_0$) cannot be a rational function, and hence that it is differentially transcendental over $C(s)$.

For $v=1$, the kernel curve defined by $xy-\frac{1}{2}\l(y^2+xy^2+x^2y^2+x^2\r)=0$ is parametrized by:
\begin{equation}\label{eq:paraEv=1}
    x_1(s)=\frac{- s^{2}}{s^{2} + 1}\quad\text{and}\quad
    y_1(s)=\frac{- s^{2}}{s^{2} + s + 1}.
\end{equation}
We have $\displaystyle{\widetilde{x}_1(s)=-\frac{s^{2}}{2 s^{2} + 2 s + 1}}$, so that
$\widetilde{x}_1(s)=x_1\l(\frac{s}{1+s}\r)$.
Therefore, the power series $G_1(s)$
in \eqref{eq:G1} satisfies the
difference equation:
    \[
     G_1(s)-G_1\l({\frac {s}{s+1}}\r)=\frac{1}{2}
     +\frac{s-1}{s^2 + s + 1}
     +\frac{1}{2(2 s^2 + 2 s + 1)}
     -\frac{s}{s^2 + 1}.
    \]
We need to prove that this functional equation has no rational solutions.
If we set $\widetilde{s}=1/s$, we get
    \[
     G_1(\widetilde s)-G_1\l(\widetilde s+1\r)=
     \frac{2 \widetilde s + 1}{\widetilde s^2 + \widetilde s + 1}
     -\frac{\widetilde s}{\widetilde s^2 + 1}
     -\frac{\widetilde s + 1}{\widetilde s^2 + 2 \widetilde s + 2}.
    \]
We define $\widetilde{G}_1(\widetilde s)=G_1(\widetilde s)+\frac{\widetilde s}{\widetilde s^{2} + 1}$, so that:
\begin{equation}\label{eq:feqEv=1}
     \widetilde G_1(\widetilde s)-\widetilde G_1\l(\widetilde s+1\r)=
     \frac{2 \widetilde s + 1}{\widetilde s^{2} + \widetilde s + 1} - \frac{2 (\widetilde s + 1)}{\widetilde s^{2} + 2 \widetilde s + 2}.
\end{equation}
The poles of the right-hand side are
    \[
    -\frac{1}2+\frac{i \sqrt{3}}{2}\,,~
    -\frac{1}2-\frac{i \sqrt{3}}{2}\,,~
    -1+i\,,~
    -1-i\,.
    \]
Since any pair of them does not differ by an integer, the functional equation above cannot admit a rational solution.
We conclude that $G_1(s)$ is differentially transcendental over $\C(\{s\})$.
\end{proof}

\def\cprime{$'$}


\begin{thebibliography}{JvRPR08}

\bibitem[AIK14]{AIK1024}
Tsuneo Arakawa, Tomoyoshi Ibukiyama, and Masanobu Kaneko.
\newblock {\em Bernoulli numbers and zeta functions}.
\newblock Springer Monographs in Mathematics. Springer, Tokyo, 2014.
\newblock With an appendix by Don Zagier.

\bibitem[Bak93]{Baker1893}
H.~F. Baker.
\newblock Examples of the application of {N}ewton's polygon to the theory of
  singular points of algebraic functions.
\newblock {\em Trans. Cambridge Philos. Soc.}, XV(IV):403--450, 1893.

\bibitem[BD81]{BarskyDumont81}
Daniel Barsky and Dominique Dumont.
\newblock Congruences pour les nombres de {G}enocchi de 2e esp\`ece.
\newblock In {\em Study {G}roup on {U}ltrametric {A}nalysis. 7th--8th years:
  1979--1981 ({P}aris, 1979/1981) ({F}rench)}. Secr\'{e}tariat Math., Paris,
  1981.
\newblock Exp. No.~34, 13 pp.

\bibitem[BDVR24]{bostan-divizio-raschel-BELL}
Alin Bostan, Lucia Di~Vizio, and Kilian Raschel.
\newblock Differential transcendence of {Bell} numbers and relatives: a
  {Galois} theoretic approach.
\newblock {\em Amer. J. Math.}, 146(6):1577--1615, 2024.

\bibitem[Bee09]{Beelen09}
Peter Beelen.
\newblock A generalization of {B}aker's theorem.
\newblock {\em Finite Fields Appl.}, 15(5):558--568, 2009.

\bibitem[BMM10]{BoMi10}
Mireille Bousquet-M\'{e}lou and Marni Mishna.
\newblock Walks with small steps in the quarter plane.
\newblock In {\em Algorithmic probability and combinatorics}, volume 520 of
  {\em Contemp. Math.}, pages 1--39. Amer. Math. Soc., Providence, RI, 2010.

\bibitem[Bon26]{Bonnet26}
Pierre Bonnet.
\newblock Walks in the quadrant with interacting boundaries: genus zero case.
\newblock Preprint, {arXiv}:2603.07827 [math.{CO}], 2026.

\bibitem[BOR19]{Beaton19}
N.~R. Beaton, A.~L. Owczarek, and A.~Rechnitzer.
\newblock Exact solution of some quarter plane walks with interacting
  boundaries.
\newblock {\em Electron. J. Combin.}, 26(3):Paper No.~3.53, 39 pp., 2019.

\bibitem[BRS14]{BoRaSa14}
Alin Bostan, Kilian Raschel, and Bruno Salvy.
\newblock Non-{D}-finite excursions in the quarter plane.
\newblock {\em J. Combin. Theory Ser. A}, 121:45--63, 2014.

\bibitem[Del01]{Dellac01}
Hippolyte Dellac.
\newblock Note sur l'\'elimination, m\'ethode de parall\'elogramme.
\newblock {\em Annales de la Facult\'e des Sciences de Marseille}, XI:141--164,
  1901.

\bibitem[DHRS18]{DHRS1}
Thomas Dreyfus, Charlotte Hardouin, Julien Roques, and Michael~F. Singer.
\newblock On the nature of the generating series of walks in the quarter plane.
\newblock {\em Invent. Math.}, 213(1):139--203, 2018.

\bibitem[DHRS20]{DHRS0}
Thomas Dreyfus, Charlotte Hardouin, Julien Roques, and Michael~F. Singer.
\newblock Walks in the quarter plane: genus zero case.
\newblock {\em J. Combin. Theory Ser. A}, 174:Paper No.~105251, 25 pp., 2020.

\bibitem[DHRS21]{DracDrey}
Thomas Dreyfus, Charlotte Hardouin, Julien Roques, and Michael~F. Singer.
\newblock On the kernel curves associated with walks in the quarter plane.
\newblock In Alin Bostan and Kilian Raschel, editors, {\em Transcendence in
  Algebra, Combinatorics, Geometry and Number Theory}, volume 373 of {\em
  Springer Proc. Math. Stat.}, pages 61--89, Cham, 2021. Springer International
  Publishing.

\bibitem[dPLW19]{nondeterministic}
\'Elie de~Panafieu, Mohamed~Lamine Lamali, and Michael Wallner.
\newblock Combinatorics of nondeterministic walks of the {D}yck and {M}otzkin
  type.
\newblock In {\em 2019 {P}roceedings of the {S}ixteenth {W}orkshop on
  {A}nalytic {A}lgorithmics and {C}ombinatorics ({ANALCO})}, pages 1--12. SIAM,
  Philadelphia, PA, 2019.

\bibitem[DZ94]{DuZe94}
Dominique Dumont and Jiang Zeng.
\newblock Further results on the {E}uler and {G}enocchi numbers.
\newblock {\em Aequationes Math.}, 47(1):31--42, 1994.

\bibitem[FI21]{FaIa-21}
Guy Fayolle and Roudolf Iasnogorodski.
\newblock Conditions for some non stationary random walks in the quarter plane
  to be singular or of genus 0.
\newblock {\em Markov Process. Related Fields}, 27(1):111--122, 2021.

\bibitem[FIM17]{FIM17}
Guy Fayolle, Roudolf Iasnogorodski, and Vadim Malyshev.
\newblock {\em Random walks in the quarter plane}, volume~40 of {\em
  Probability Theory and Stochastic Modelling}.
\newblock Springer, Cham, second edition, 2017.
\newblock Algebraic methods, boundary value problems, applications to queueing
  systems and analytic combinatorics.

\bibitem[Gou72]{Gould72}
H.~W. Gould.
\newblock Explicit formulas for {B}ernoulli numbers.
\newblock {\em Amer. Math. Monthly}, 79:44--51, 1972.

\bibitem[HRT23]{HoRaTa-23}
Viet~Hung Hoang, Kilian Raschel, and Pierre Tarrago.
\newblock Harmonic functions for singular quadrant walks.
\newblock {\em Indag. Math. (N.S.)}, 34(5):936--972, 2023.

\bibitem[Ish98]{Ishizaki-hypertransc}
Katsuya Ishizaki.
\newblock Hypertranscendency of meromorphic solutions of a linear functional
  equation.
\newblock {\em Aequationes Math.}, 56(3):271--283, 1998.

\bibitem[JvRPR08]{Prellberg-08}
E.~J. Janse~van Rensburg, T.~Prellberg, and A.~Rechnitzer.
\newblock Partially directed paths in a wedge.
\newblock {\em J. Combin. Theory Ser. A}, 115(4):623--650, 2008.

\bibitem[Kam89]{ka89}
U.~Kamps.
\newblock Chebyshev polynomials and least squares estimation based on
  one-dependent random variables.
\newblock {\em Linear Algebra Appl.}, 112:217--230, 1989.

\bibitem[Kla03]{Klazar03}
Martin Klazar.
\newblock Bell numbers, their relatives, and algebraic differential equations.
\newblock {\em J. Combin. Theory Ser. A}, 102(1):63--87, 2003.

\bibitem[Mis09]{Mishna-09}
Marni Mishna.
\newblock Classifying lattice walks restricted to the quarter plane.
\newblock {\em J. Combin. Theory Ser. A}, 116(2):460--477, 2009.

\bibitem[MM14]{MeMi14}
Stephen Melczer and Marni Mishna.
\newblock Singularity analysis via the iterated kernel method.
\newblock {\em Combin. Probab. Comput.}, 23(5):861--888, 2014.

\bibitem[MR09]{MiRe09}
Marni Mishna and Andrew Rechnitzer.
\newblock Two non-holonomic lattice walks in the quarter plane.
\newblock {\em Theoret. Comput. Sci.}, 410(38--40):3616--3630, 2009.

\bibitem[Oga14]{Ogawara-formalIshizaki}
Hiroshi Ogawara.
\newblock Differential transcendency of a formal {L}aurent series satisfying a
  rational linear {$q$}-difference equation.
\newblock {\em Funkcial. Ekvac.}, 57(3):477--488, 2014.

\bibitem[Pak18]{Pak18}
Igor Pak.
\newblock Complexity problems in enumerative combinatorics.
\newblock In {\em Proceedings of the {I}nternational {C}ongress of
  {M}athematicians---{R}io de {J}aneiro 2018. {V}ol. {IV}. {I}nvited lectures},
  pages 3153--3180. World Sci. Publ., Hackensack, NJ, 2018.

\bibitem[Poz11]{Po-11}
Svetlana Poznanovi{\'c}.
\newblock A bijection between partially directed paths in the symmetric wedge
  and matchings.
\newblock {\em Ann. Comb.}, 15(2):331--339, 2011.

\bibitem[Shi15]{Shi}
Zhan Shi.
\newblock {\em Branching random walks}, volume 2151 of {\em Lecture Notes in
  Mathematics}.
\newblock Springer, Cham, 2015.
\newblock Lecture notes from the 42nd Probability Summer School held in Saint
  Flour, 2012, \'Ecole d'\'Et\'e{} de Probabilit\'es de Saint-Flour.
  [Saint-Flour Probability Summer School].

\bibitem[Sta12]{St-12}
Richard~P. Stanley.
\newblock {\em Enumerative combinatorics. {V}olume 1}, volume~49 of {\em
  Cambridge Studies in Advanced Mathematics}.
\newblock Cambridge University Press, Cambridge, second edition, 2012.

\bibitem[Woe00]{Wo00}
Wolfgang Woess.
\newblock {\em Random walks on infinite graphs and groups}, volume 138 of {\em
  Cambridge Tracts in Mathematics}.
\newblock Cambridge University Press, Cambridge, 2000.

\end{thebibliography}
\end{document}